\documentclass[smallextended]{svjour3}
\RequirePackage{fix-cm}
\RequirePackage{amsmath}
\smartqed  
\usepackage[T1]								{fontenc}
\usepackage									{amssymb, amsmath}
\usepackage									{bm, bbm}
\usepackage[fixamsmath]						{mathtools}
\usepackage{color}

\DeclareMathOperator	{\IN}		{\mathbb{N}} 	
	\DeclareMathOperator	{\IR}		{\mathbb{R}}
	\DeclareMathOperator	{\IE}		{\mathbb{E}} 
	\DeclareMathOperator	{\IP}		{\mathbb{P}}

\DeclareMathOperator	{\Var}		{Var}

\DeclareMathOperator	{\Ent}		{Ent}

\DeclareMathOperator	{\Id}		{Id}

\DeclareMathOperator	{\bb}		{\boldsymbol{\beta}}

\DeclarePairedDelimiter	{\abs}		{\lvert}	{\rvert}
\DeclarePairedDelimiter	{\norm}		{\lVert}	{\rVert}
\DeclarePairedDelimiter	{\skal}		{\langle}	{\rangle}

\newcommand				{\eins}		{\text{$\mathbbm{1}$}}
\newcommand				{\dpartial}	{\mathfrak{d}}

\newcommand				{\bnorm}[1]	{\bigg\lVert #1 \bigg\rVert}
\renewcommand 			{\epsilon}	{\varepsilon}
\renewcommand			{\phi}		{\varphi}
\renewcommand			{\tilde}	{\widetilde}

\allowdisplaybreaks

\begin{document}
\title{Concentration inequalities for bounded functionals via log-Sobolev-type inequalities\thanks{This research was supported by the German Research Foundation (DFG) via CRC 1283 ``\textit{Taming uncertainty and profiting from randomness and low regularity in analysis, stochastics and their applications}''.}
}

\titlerunning{Concentration inequalities for bounded functionals}        

\author{Friedrich G{\"o}tze         \and
	Holger Sambale \and
	Arthur Sinulis$^*$ 
}


\institute{Friedrich G{\"o}tze \and Holger Sambale \and Arthur Sinulis \at Fakult{\"a}t f{\"u}r Mathematik\\Universit{\"a}t Bielefeld\\ Postfach 10 01 31\\ 33501 Bielefeld\\ Germany\\
	\email{goetze@math.uni-bielefeld.de} \\ \email{hsambale@math.uni-bielefeld.de} \\ \email{asinulis@math.uni-bielefeld.de}
}

\date{Received: date / Accepted: date}

\maketitle

\begin{abstract}
	In this paper we prove multilevel concentration inequalities for bounded functionals $f = f(X_1, \ldots, X_n)$ of random variables $X_1, \ldots, X_n$ that are either independent or satisfy certain logarithmic Sobolev inequalities. The constants in the tail estimates depend on the operator norms of $k$-tensors of higher order differences of $f$.
	
	We provide applications for both dependent and independent random variables. This includes deviation inequalities for empirical processes $f(X) = \sup_{g \in \mathcal{F}} \abs{g(X)}$ and suprema of homogeneous chaos in bounded random variables in the Banach space case $f(X) = \sup_{t} \norm{\sum_{i_1 \neq \ldots \neq i_d} t_{i_1 \ldots i_d} X_{i_1} \cdots X_{i_d}}_{\mathcal{B}}$. The latter application is comparable to earlier results of Boucheron--Bousquet--Lugosi--Massart and provides the upper tail bounds of Talagrand. In the case of Rademacher random variables, we give an interpretation of the results in terms of quantities familiar in Boolean analysis. Further applications are concentration inequalities for $U$-statistics with bounded kernels $h$ and for the number of triangles in an exponential random graph model.
	\subclass{60E15 \and 05C80}
\end{abstract}
\maketitle

\section{Introduction}				\label{section:Introduction}
During the last forty years, the \emph{concentration of measure phenomenon} has become an established part of probability theory with applications in numerous fields, as is witnessed by the monographs \cite{MS86,Led01,BLM13,RS14,vH16}. One way to prove concentration of measure is by using functional inequalities, more specifically the \emph{entropy method}. It has emerged as a way to prove several groundbreaking concentration inequalities in product spaces by Talagrand \cite{Tal91,Tal96a}, mainly in the works \cite{Led97} and \cite{BL97}, and further developed in \cite{Ma00}. \par
To convey the idea, let us recall that the \emph{logarithmic Sobolev inequality} for the standard Gaussian measure $\nu$ in $\IR^n$ (see \cite{Gr75}) states that for any $f \in C_c^\infty(\IR^n)$ we have
\begin{align}		\label{eqn:LSIgross}
\Ent_{\nu}(f^2) \le 2 \int \abs{\nabla f}^2 d\nu,
\end{align}
where $\Ent_{\nu}(f^2) = \int f^2 \log f^2 d\nu - \int f^2 d\nu \log \int f^2 d\nu$ is the \emph{entropy functional}. Informally, it bounds the disorder of a function $f$ (under $\nu$) by its average local fluctuations, measured in terms of the length of the gradient. It is by now standard that \eqref{eqn:LSIgross} implies subgaussian tail decay for Lipschitz functions (e.\,g. by means of the Herbst argument). In particular, if $f \colon \mathbb{R}^n \to \mathbb{R}$ is a $\mathcal{C}^1$ function such that $|\nabla f| \le L$ a.s., we have $\nu(|f - \int f d\nu| \ge t) \le 2 \exp(-t^2/(2L^2))$ for any $t \ge 0$.

If $\mu$ is a probability measure on a discrete set $\mathcal{X}$ (or a more abstract set not allowing for an immediate replacement for $\abs{\nabla f}$), then there are several ways to reformulate equation \eqref{eqn:LSIgross}, see e.\,g. \cite{DSC96} or \cite{BT06}. We continue these ideas by working in the framework of \emph{difference operators}. Given a probability space $(\mathcal{Y}, \mathcal{A}, \mu)$, we call any operator $\Gamma: L^\infty(\mu) \to L^\infty(\mu)$ satisfying $|\Gamma (af + b)| = a\, |\Gamma f|$ for all $a > 0$, $b \in \mathbb{R}$ a difference operator. Accordingly, we say that $\mu$ satisfies a $\Gamma\mathrm{-LSI}(\sigma^2)$, if for all bounded measurable functions $f$ we have
\begin{align}	\label{eqn:defiLSImitGamma}
	\Ent_{\mu}(f^2) \le 2\sigma^2 \int \Gamma(f)^2 d\mu.
\end{align}
Apart from the domain of $\Gamma$, it is clear that \eqref{eqn:defiLSImitGamma} can be seen as generalization of \eqref{eqn:LSIgross} by defining $\Gamma(f) = \abs{\nabla f}$ on $\IR^n$.

Another route to obtain concentration inequalities is to modify the entropy method, which was done in the framework of so-called $\phi$-entropies. The idea is to replace the function $\phi_0(x) \coloneqq x\log x$ in the definition of the entropy $\Ent^{\phi_0}_{\mu}(f) = \IE_\mu \phi_0(f) - \phi_0(\IE_\mu f)$ by other functions $\phi$. This has been studied in \cite{LO00,BLM03,Cha04}. In the seminal work \cite{BBLM05} the authors proved inequalities for $\phi$-entropies for power functions $\phi(x) = \abs{x}^\alpha, \alpha \in (1,2]$, leading to moment inequalities for independent random variables.

Originally, the entropy method was primarily used to prove sub-Gaussian concentration inequalities for Lipschitz-type functions. However, there are many situations of interest in which the functions under consideration are not Lipschitz or have Lipschitz constants which grow as the dimension increases even after a renormalization which asymptotically stabilizes the variance. Among the simplest examples are polynomial-type functions. Here, the boundedness of the gradient typically has to be replaced by more elaborate conditions on higher order derivatives (up to some order $d$). Moreover, we cannot have subgaussian tail decay anymore. This is already obvious if we consider the product of two independent standard normal random variables, which leads to subexponential tails. We refer to this topic as \emph{higher order concentration}.

The earliest higher order concentration results date back to the late 1960s. Already in \cite{Bo68,Bo70} and \cite{Ne73}, the growth of $L^p$ norms and hypercontractive estimates of polynomial-type functions in Rademacher or Gaussian random variables respectively have been studied. The question of estimating the growth of $L^p$ norms of multilinear polynomials in Gaussian random variables was considered in \cite{Bor84}, \cite{AG93} and \cite{La06}. In the context of Erd{\"o}s--R{\'e}nyi graphs and the triangle problem, concentration inequalities for polynomials functions gained considerable attention, in papers such as \cite{KV00}.

More recently, \emph{multilevel concentration inequalities} have been proven in \cite{Ad06,Wo13,AW15} for many classes of functions. These included $U$-statistics in independent random variables, functions of random vectors satisfying Sobolev-type inequalities and polynomials in sub-Gaussian random variables respectively. We refer to inequalities of the type
\begin{equation}\label{MLI}
\IP\Big( \abs{f(X) - \IE f(X)} \ge t \Big) \le 2 \exp\Big( - \frac{1}{C} \min_{k =1,\ldots,d} f_k(t) \Big)
\end{equation}
as multilevel or \emph{higher order ($d$-th order) concentration  inequalities}. This means that the tails might have different decay properties in some regimes of $[0,\infty)$. Usually, we have $f_k(t) = (t/C_k)^{2/k}$ for some constant $C_k$ which typically depends on the $k$-th order derivatives.

To convey the basic idea of multilevel concentration inequalities, let us once again consider the case $d=2$, e.\,g. a quadratic form of independent, say, Gaussian random variables. As sketched above, in this case the tails decay subexponentially in general. By means of a multilevel concentration inequality (the so-called Hanson--Wright inequality, which we address in more detail at a later point), we can show that while for $t$ large, subexponential tail decay holds, for small $t$ we even get subgaussian decay. In this sense, multilevel concentration inequalities provide refined tail estimates which do not only cover the behavior for large $t$.

Our own work started with a second order concentration inequality on the sphere in \cite{BCG17} and was continued in \cite{BGS18} for bounded functionals of various classes of random variables (e.\,g. independent random variables or in presence of a logarithmic Sobolev inequality \eqref{eqn:LSIgross}), and in \cite{GSS18} for weakly dependent random variables (e.\,g. the Ising model). In these papers, we studied higher order concentration, arriving at multi-level tail inequalities of type \eqref{MLI}. If the underlying measure $\mu$ satisfies a logarithmic Sobolev inequality, \cite[Corollary 1.11]{BGS18} yields $f_k(t) = (t/C_k)^{2/k}$ with $C_k = (\int |f^{(k)}|_\mathrm{op}^2 d\mu)^{1/2}$ for $k = 1, \ldots, d-1$ and $C_d = \sup |f^{(d)}|_\mathrm{op}$, where $\abs{f^{(k)}}_\mathrm{op}$ denotes the operator norm of the respective tensors of $k$-th order partial derivatives. A downside in both \cite{BGS18} and \cite{GSS18} is that for functions of independent or weakly dependent random variables, comparable estimates involve Hilbert--Schmidt instead of operator norms, leading to weaker estimates in general.

A central aspect of the present article is to fix this drawback by a slightly more elaborate approach. Here, we consider both independent and dependent random variables. In either case, we prove multilevel concentration inequalities of the same type, and apply them to different forms of functionals. We provide improvements of earlier higher order concentration results like \cite[Theorem 1.1]{BGS18} or \cite[Theorem 1.5]{GSS18}, replacing the Hilbert--Schmidt norms appearing therein by operator norms. This leads to sharper bounds and a wider range of applicability.

A special emphasis is placed on providing \emph{uniform versions} of the higher order concentration inequalities. By this, we mean that we consider functionals of supremum type $f(X) = \sup_{f \in \mathcal{F}} \abs{f(X)}$, which includes suprema of polynomial chaoses, or empirical processes.
Two more applications are given by $U$-statistics in independent and weakly dependent random variables as well as a triangle counting statistic in some models of random graphs, for which we prove concentration inequalities.

\textbf{Notations.} Throughout this note, $X = (X_1,\ldots,X_n)$ is a random vector taking values in some product space $\mathcal{Y} = \otimes_{i = 1}^n \mathcal{X}_i$ (equipped with the product $\sigma$-algebra) with law $\mu$, defined on a probability space $(\Omega, \mathcal{A}, \mathbb{P})$. By abuse of language, we say that $X$ satisfies a $\Gamma\mathrm{-LSI}(\sigma^2)$, if its distribution does. In any finite-dimensional vector space, we let $\abs{\cdot}$ be the Euclidean norm, and for brevity, we write $[q] \coloneqq \{1,\ldots, q\}$ for any $q \in \IN$. Given a vector $x = (x_j)_{j = 1,\ldots,n}$ we write $x_{i^c} = (x_j)_{j \neq i}$. To any $d$-tensor $A$ we define the Hilbert--Schmidt norm $\abs{A}_\mathrm{HS} \coloneqq (\sum_{i_1, \ldots, i_d} A_{i_1 \ldots i_d}^2)^{1/2}$ and the operator norm
\begin{align*}				
	\abs{A}_\mathrm{\mathrm{op}} \coloneqq \sup_{\substack{v^1, \ldots, v^d \in \IR^n \\ \abs{v^j} \le 1}} \skal{v^1 \cdots v^d, A} = \sup_{\substack{v^1, \ldots, v^d \\ \abs{v^j} \le 1}} \sum_{i_1,\ldots,i_d} v^1_{i_1} \cdots v^d_{i_d} A_{i_1 \ldots i_d},
\end{align*}
using the outer product $(v^1 \cdots v^d)_{i_1 \ldots i_d} = \prod_{j = 1}^d v^j_{i_j}$.
For brevity, for any random $k$-tensor $A$ and any $p \in (0,\infty]$ we abbreviate $\norm{A}_{\mathrm{HS},p} = (\IE\abs{A}_{\mathrm{HS}}^p)^{1/p}$ as well as $\norm{A}_{\mathrm{op},p} = (\IE\abs{A}_{\mathrm{op}}^p)^{1/p}$.
Lastly, we ignore any measurability issues that may arise. Thus, we assume that all the suprema used in this work are either countable or defined as $\sup_{t \in T} = \sup_{F \subset T : F \text{ finite}} \sup_{t \in F}$.

\subsection{Main results}
To formulate our main results, we introduce a difference operator labeled $\abs{\mathfrak{h} f}$ which is frequently used in the method of bounded differences. Let $X' = (X_1',\ldots, X'_n)$ be an independent copy of $X$, defined on the same probability space. Given $f(X) \in L^\infty(\mathbb{P})$, define for each $i \in [n]$
\begin{equation*}
T_i f \coloneqq T_if(X) \coloneqq f(X_{i^c}, X_i') = f(X_1, \ldots, X_{i-1}, X_i',\linebreak[2] X_{i+1},\ldots, X_n)
\end{equation*}
and
\begin{equation*}	
\mathfrak{h}_i f(X) = \lVert f(X) - T_if(X) \rVert_{i, \infty},
\qquad \mathfrak{h} f(X) = (\mathfrak{h}_1 f(X), \ldots, \mathfrak{h}_n f(X)),
\end{equation*}
where $\lVert \cdot \rVert_{i, \infty}$ denotes the $L^\infty$-norm with respect to $(X_i,X_i')$. The difference operator $\abs{\mathfrak{h}f}$ is given as the Euclidean norm of the vector $\mathfrak{h} f$.

We shall also need higher order versions of $\mathfrak{h}$, denoted by $\mathfrak{h}^{(d)} f$. They can be thought of as analogues of the $d$-tensors of all partial derivatives of order $d$ in an abstract setting. To define the $d$-tensor $\mathfrak{h}^{(d)} f$, we specify it on its ``coordinates''. That is, given distinct indices $i_1, \ldots, i_d$, we set
\begin{align} 		\label{eqn:DefiHigherOrderDerivatives}
\begin{split}
\mathfrak{h}_{i_1 \ldots i_d}f(X) 
  = \; &
\Big\lVert\, \prod_{s=1}^d\, 
(\mathrm{Id} - T_{i_s})f(X) \Big\rVert_{i_1, \ldots, i_d, \infty}\\
   = \; &
\Big\lVert\, f(X) + \sum_{k=1}^d\, (-1)^k \sum_{1 \leq s_1 < \ldots < s_k \leq d}
T_{i_{s_1} \ldots i_{s_k}}f(X)\, \Big\rVert_{i_1, \ldots, i_d, \infty}
\end{split}
\end{align}
where $T_{i_1 \ldots i_d} = T_{i_1} \circ \ldots \circ T_{i_d}$ exchanges the random variables $X_{i_1}, \ldots, X_{i_d}$ by $X'_{i_1}, \ldots, X'_{i_d}$, and 
$\lVert \cdot \rVert_{i_1, \ldots, i_d, \infty}$ denotes the $L^\infty$-norm 
with respect to the random variables $X_{i_1}, \ldots, X_{i_d}$ and 
$X_{i_1}', \ldots, X_{i_d}'$. For instance, for $i \ne j$,  
\[
	\mathfrak{h}_{ij} f(X) = \lVert f(X) - T_if(X) - T_jf(X) + T_{ij}f(X) \rVert_{i,j,\infty}.
\]
Using the definition \eqref{eqn:DefiHigherOrderDerivatives}, we define tensors of $d$-th order differences as follows:
\begin{align*} 
\big(\mathfrak{h}^{(d)}f(X)\big)_{i_1 \ldots i_d} = 
\begin{cases} 
\mathfrak{h}_{i_1 \ldots i_d}f(X), & 
\text{if $i_1, \ldots, i_d$ are distinct}, \\ 0, & \text{else}. 
\end{cases}
\end{align*}
Whenever no confusion is possible, we omit writing the random vector $X$, i.\,e. we freely write $f$ instead of $f(X)$ and $\mathfrak{h}^{(d)}f$ instead of $\mathfrak{h}^{(d)}f(X)$.

Our first main theorem is a concentration inequality for general, bounded functionals of independent random variables $X_1, \ldots, X_n$. 

\begin{theorem}			\label{theorem:Lpnormestimates}
	Let $X$ be a random vector with independent components, $f: \mathcal{Y} \to \IR$ a measurable function satisfying $f = f(X) \in L^\infty(\IP)$, $d \in \IN$ and define $C \coloneqq 217d^2$. We have for any $t \ge 0$
	{\footnotesize
	\begin{align}	\label{eqn:LpNormEstimatesIndependentOnly}
		\IP\left( \abs{f - \IE f} \ge t \right) \le 2 \exp \Big( -\frac{1}{C} \min_{k = 1,\ldots,d-1} \Big( \frac{t}{\norm{\mathfrak{h}^{(k)} f}_{\mathrm{op},1}} \Big)^{2/k} \wedge \Big(\frac{t}{\norm{\mathfrak{h}^{(d)} f}_{\mathrm{op},\infty}}\Big)^{2/d} \Big).
	\end{align}
	}
\end{theorem}

For the sake of illustration, let us consider the case of $d=2$. Assuming that $X_1, \ldots, X_n$ satisfy $\mathbb{E}X_i = 0$, $\mathbb{E} X_i^2 = 1$ and $\abs{X_i} \le M$ a.s., let $f(X)$ be the quadratic form $f(X) = \sum_{i<j} a_{ij} X_i X_j = X^TAX$. Here, $a_{ij} \in \mathbb{R}$ for all $i<j$, and $A$ is the symmetric matrix with zero diagonal and entries $A_{ij} = a_{ij}/2$ if $i < j$. In this case, it is easy to see that $\norm{\mathfrak{h} f}_{\mathrm{op},1} \le \norm{\mathfrak{h} f}_{\mathrm{op},2} \le 4M \abs{A}_\mathrm{HS}$ and $\norm{\mathfrak{h}^{(2)} f}_{\mathrm{op},\infty} \le 8M^2 \abs{A^\mathrm{abs}}_\mathrm{op}$, where $A^\mathrm{abs}$ is the matrix given by $(A^\mathrm{abs})_{ij} = \abs{a_{ij}}$. As a result,
\[
\IP\left( \abs{f - \IE f} \ge t \right) \le 2 \exp \Big( -\frac{1}{CM^2} \min \Big(\frac{t^2}{\abs{A}_\mathrm{HS}^2}, \frac{t}{\abs{A^\mathrm{abs}}_\mathrm{op}} \Big) \Big).
\]
This is a version of the famous \emph{Hanson--Wright inequality}. For the various forms of the Hanson--Wright inequality we refer to \cite{HW71,W73,HKZ12,RV13,VW15,Ad15,ALM18}.

Note that by a modification of our proofs (using arguments especially adapted to polynomials), it is possible to replace $\abs{A^\mathrm{abs}}_\mathrm{op}$ by $\abs{A}_\mathrm{op}$, thus avoiding the drawback of switching to a matrix with a possibly larger operator norm. See Section \ref{subsection:UnifBds} and \ref{subsection:Polyn+subgrc} for details. On the other hand, Theorem \ref{theorem:Lpnormestimates} allows for \emph{any} function $f$, not just quadratic forms, and the case of $d=2$ can in this sense be considered as generalization of the Hanson--Wright inequality.

For a certain class of weakly dependent random variables $X_1, \ldots, X_n$, we can prove similar estimates as in Theorem \ref{theorem:Lpnormestimates}.
To this end, we introduce another difference operator, which is more familiar in the context of logarithmic Sobolev inequalities for Markov chains as developed in \cite{DSC96}. Assume that $\mathcal{Y} = \otimes_{i = 1}^n \mathcal{X}_i$ for some finite sets $\mathcal{X}_1, \ldots, \mathcal{X}_n$, equipped with a probability measure $\mu$ and let $\mu(\cdot \mid x_{i^c})$ denote the conditional measure (interpreted as a measure on $\mathcal{X}_i$) and $\mu_{i^c}$ the marginal on $\otimes_{j \ne i} \mathcal{X}_j$. Finally, set
\begin{align*}
	\abs{\dpartial f}^2(x) &\coloneqq \sum_{i = 1}^n (\dpartial_i f(x))^2 \coloneqq \sum_{i = 1}^n \mathrm{Var}_{\mu(\cdot \mid x_{i^c})}(f(x_{i^c}, \cdot)) \\
  &=\sum_{i = 1}^n \frac{1}{2} \iint (f(x_{i^c},y) - f(x_{i^c},y'))^2 d\mu(y \mid x_{i^c}) d\mu(y' \mid x_{i^c}).
\end{align*}
This difference operator appears naturally in the Dirichlet form associated to the Glauber dynamic of $\mu$, given by
\begin{align*}	
	\mathcal{E}(f,f) \coloneqq \sum_{i = 1}^n \int \Var_{\mu(\cdot \mid x_{i^c})}(f(x_{i^c},\cdot)) d\mu_{i^c}(x_{i^c}) = \int \abs{\dpartial f}^2 d\mu.
\end{align*}
In the next theorem, we require a $\dpartial$--LSI for the underlying random variables $X_1, \ldots, X_n$. A number of models which satisfy this assumption will be discussed below.

\begin{theorem}			\label{theorem:partialLSITails}
Let $X = (X_1, \ldots, X_n)$ be a random vector satisfying a $\dpartial\mathrm{-LSI}(\sigma^2)$ and $f: \mathcal{Y} \to \IR$ a measurable function with $f = f(X) \in L^\infty(\IP)$. With the constant $C = 15 \sigma^2 d^2 > 0$ we have for any $t \ge 0$
{\footnotesize
	\begin{align}	\label{eqn:LpestimatesWithOPNorms}
		\IP\left( \abs{f - \IE f} \ge t \right) \le 2 \exp \Big( -\frac{1}{C} \min_{k = 1,\ldots,d-1} \Big( \frac{t}{\norm{\mathfrak{h}^{(k)} f}_{\mathrm{op},1}} \Big)^{2/k} \wedge \Big(\frac{t}{\norm{\mathfrak{h}^{(d)} f}_{\mathrm{op},\infty}}\Big)^{2/d} \Big).
	\end{align}
}
\end{theorem}

Again, if $d=2$, assuming that $\mathbb{E}X_i = 0$, $\mathbb{E} X_i^2 = 1$, $\abs{X_i} \le M$ a.s. and $\mathbb{E}X_iX_j = 0$ if $i \ne j$, we arrive at a Hanson--Wright type inequality, this time including dependent situations. Similar results still hold if we remove the uncorrelatedness condition.

Let us discuss the $\mathfrak{d}$--LSI condition in more detail. First, any collection of random independent variables $X_1, \ldots, X_n$ with finitely many values satisfies a $\dpartial\mathrm{-LSI}(\sigma^2)$ with $\sigma^2$ depending on the minimal non-zero probability of the $X_i$ (cf. Proposition \ref{d-LSIs}). In this situation, Theorem \ref{theorem:Lpnormestimates} and Theorem \ref{theorem:partialLSITails} only differ by constants.

However, the $\dpartial$--LSI conditions also gives rise to numerous models of dependent random variables as in \cite[Proposition 1.1]{GSS18} (the Ising model) or \cite[Theorem 3.1]{SS18} (various different models). Let us recall some of them. The \emph{Ising model} is the probability measure on $\{\pm1\}^n$ defined by normalizing $\pi(\sigma) = \exp( \frac{1}{2} \sum_{i,j} J_{ij} \sigma_i \sigma_j + \sum_{i = 1}^n h_i \sigma_i )$ for a symmetric matrix $J = (J_{ij})$ with zero diagonal and some $h \in \mathbb{R}^n$. In \cite[Proposition 1.1]{GSS18}, we have shown that if $\max_{i = 1,\ldots,n} \sum_{j = 1}^n \abs{J_{ij}} \le 1-\alpha$ and $\max_{i \in [n]} \abs{h_i} \le \tilde{\alpha}$, the Ising model satisfies a $\dpartial\mathrm{-LSI}(\sigma^2)$ with $\sigma^2$ depending on $\alpha$ and $\tilde{\alpha}$ only. For the special case of $h = 0$ and $J_{ij} = \beta$ for all $i \ne j$, we obtain the \emph{Curie--Weiss model}. Here, the two conditions required above reduce to $\beta < 1$.

Another simple model in which a $\dpartial$--LSI holds is the \emph{random coloring model}. If $G = (V,E)$ is a finite graph and $C = \{1, \ldots, k\}$ is a set of colors, we denote by $\Omega_0 \subset C^V$ the set of all proper coloring, i.\,e. the set of all $\omega \in C^V$ such that $\{v,w\} \in E \Rightarrow \omega_v \ne \omega_w$. In \cite[Theorem 3.1]{SS18}, we have shown that the uniform distribution on $\Omega_0$ satisfies a $\dpartial$--LSI if the maximum degree $\Delta$ is uniformly bounded and $k \ge 2\Delta+1$ (strictly speaking, we consider sequences of graphs here). In \cite[Theorem 3.1]{SS18}, we moreover prove $\dpartial$--LSIs for the (vertex-weighted) \emph{exponential random graph model} and the \emph{hard-core model}. We will further discuss the exponential random graph model in Section \ref{subsection:Polyn+subgrc}.

The common feature in all these models is that the dependencies which appear can be controlled (e.\,g. by means of a coupling matrix which measures the interactions between the particles of the system under consideration, cf. \cite[Theorem 4.2]{GSS18}) in such a way that the model is not ``too far'' from a product measure. For instance, in the Curie--Weiss model, this just translates to $\beta < 1$.

As a final remark, we discuss the LSI property with respect to various difference operators in Section \ref{section:LSIs&DOs}. In particular, we show that the restriction to finite spaces which is implicit in Theorem \ref{theorem:partialLSITails} is natural since the $\dpartial\mathrm{-LSI}$ property requires the underlying space to be finite. By contrast, we prove that any set of independent random variables $X_1, \ldots, X_n$ satisfies an $\mathfrak{h}$--LSI$(1)$. However, it seems that it is not possible to use the entropy method based on $\mathfrak{h}$--LSIs.

The upper bound in Theorem \ref{theorem:partialLSITails} admits a ``uniform version'', i.\,e. we can prove deviation inequalities for suprema of functions, in the following sense. Let $\mathcal{F}$ be a family of uniformly bounded, real-valued, measurable functions and set
\begin{align}\label{eqn:gF}
g(X) \coloneqq g_{\mathcal{F}}(X) \coloneqq \sup_{f \in \mathcal{F}} \abs{f(X)}.
\end{align}

For any $d \in \IN$ and $j = 1,\ldots, d$ let $W_j = W_j(X) \coloneqq \sup_{f \in \mathcal{F}} \abs{\mathfrak{h}^{(j)} f(X)}_{\mathrm{op}}$.

\begin{theorem}\label{theorem:SupremaOfFunctions}
Assume that either $X_1, \ldots, X_n$ are independent or $X$ satisfies a $\dpartial\mathrm{-LSI}(\sigma^2)$ and let $g = g(X)$ be as in \eqref{eqn:gF}. With the same constant $C$ as in Theorem \ref{theorem:partialLSITails} or \ref{theorem:Lpnormestimates} respectively, we have for any $t \ge 0$ the deviation inequality
\[
\IP(g - \IE g \ge t) \le 2 \exp\Big( - \frac{1}{C} \min\Big( \min_{j = 1,\ldots, d-1} \Big( \frac{t}{\IE W_j} \Big)^{2/k}, \frac{t^{2/d}}{\norm{W_d}_\infty} \Big) \Big)
\]
\end{theorem}


As mentioned before, Theorem \ref{theorem:SupremaOfFunctions} yields bounds for the upper tail only. The background is that the entropy method has certain limitations when it is applied to suprema of functions, cf. also Proposition \ref{prop:SupremaOfSums} or Theorem \ref{theorem:ChaosInIndependentOrPartialLSI} below. Roughly sketched, the reason is that when evaluating difference operators of suprema, if a positive part is involved we may typically choose a coordinate-independent maximizer of the terms involved. Without a positive part, this is no longer possible. See in particular the proof of Theorem \ref{theorem:ChaosInIndependentOrPartialLSI}, where we provide some further details.

Functionals of the form \eqref{eqn:gF} have been considered in various works, starting from the first results in \cite[Theorem 1.4]{Tal96a}, and continued in \cite[Th{\'e}or{\`e}me 1.1]{Rio02}, \cite[Theorem 3]{Ma00} and \cite[Theorem 2.3]{Bo02} in the special case of
\begin{equation}\label{eqn:SupremaOfSums}
  g(X) \coloneqq \sup_{f \in \mathcal{F}} \Big\lvert \sum_{j = 1}^n f(X_j) \Big\rvert.
\end{equation}
Further research has been done in \cite{KR05}, \cite[Section 3]{Sam07} and more recently \cite[Proposition 5.4]{Mar18}. In these works, Bennett-type inequalities have been proven for general independent random variables. Furthermore, \cite[Theorem 10]{BBLM05} treats the case $g(X) = \sup_{t \in \mathcal{T}} \sum_{i = 1}^n t_i X_i$ for Rademacher random variables $X_i$ and a compact set of vectors $\mathcal{T} \subset \IR^n$. As a byproduct of our method, we prove a deviation inequality for $g$ which can be regarded as a uniform bounded differences inequality.

\begin{proposition}\label{prop:SupremaOfSums}
Assume that $X = (X_1, \ldots, X_n)$ satisfies a $\dpartial\mathrm{-LSI}(\sigma^2)$, let $g = g(X)$ be as in \eqref{eqn:SupremaOfSums}, and let $c(f)$ be such that $\abs{f(x) - f(y)} \le c(f)$. For any $t \ge 0$ we have
\[
\IP\Big(g \ge \IE g + t\Big) \le 2 \exp\Big( - \frac{t^2}{15 \sigma^2 n \sup_{f \in \mathcal{F}} c(f)^2} \Big). 
\]
\end{proposition}

Let us put Proposition \ref{prop:SupremaOfSums} into context. In the above mentioned works, the authors derive Bennett-type inequalities for independent random variables $X_1, \ldots, X_n$, whereas in our case the concentration inequalities have sub-Gaussian tails. It might be compared to the sub-Gaussian tail estimates for Bernoulli processes, see e.\,g. \cite[Theorem 5.3.2]{Tal14}. However, the $\dpartial\mathrm{-LSI}(\sigma^2)$ property is both more and less general. On the one hand, it is possible to include possibly dependent random vectors, but on the other hand for independent random variables it is only applicable if the $X_i$ take finitely many values.


\subsection{Outline}
In Section \ref{section:Applications}, we present a number of applications and refinements of our main results. Section \ref{section:concentrationIneqUnderLSI} contains the proofs of our main theorems. The proofs of the results from Section \ref{section:Applications} is deferred to Section \ref{section:SupremaOfChaos}. We close out the paper by discussing different forms of logarithmic Sobolev inequalities with respect to various difference operators in the last Section \ref{section:LSIs&DOs}.

\section{Applications}\label{section:Applications}

In the sequel, we consider various situations in which our results can be applied. Some of them can be regarded as sharpenings of our main theorems for functions which have a special structure.

\subsection{Uniform bounds}\label{subsection:UnifBds}
If the functions under consideration are of polynomial type, we may somewhat refine the results from the previous section. Here we focus on uniform bounds as discussed in Theorem \ref{theorem:SupremaOfFunctions}.

Let $\mathcal{I}_{n,d}$ denote the family of subsets of $[n]$ with $d$ elements, fix a Banach space $(\mathcal{B}, \norm{\cdot})$ with its dual space $(\mathcal{B}^*, \norm{\cdot}_*)$, a compact subset $\mathcal{T} \subset \mathcal{B}^{\mathcal{I}_{n,d}}$ and let $\mathcal{B}_1^*$ be the $1$-ball in $\mathcal{B}^*$ with respect to $\norm{\cdot}_*$. Let $X = (X_1, \ldots, X_n)$ be a random vector with support in $[a,b]^n$ for some real numbers $a < b$ and define
\begin{align}		\label{eqn:fTauNew}
	f(X) \coloneqq f_{\mathcal{T}}(X) \coloneqq \sup_{t \in \mathcal{T}}  \Big \lVert \sum_{I \in \mathcal{I}_{n,d}} X_I t_I \Big\rVert,
\end{align} 
where $X_I \coloneqq \prod_{i \in I} X_i$. For any $k \in [d]$ we let
\begin{align}	\label{eqn:Wkallgemein}
\begin{split}
	W_k &\coloneqq \sup_{t \in \mathcal{T}} \sup_{v^* \in \mathcal{B}_1^*} \sup_{\substack{\alpha^1, \ldots, \alpha^k \in \IR^n \\ \abs{\alpha^i} \le 1}} v^*\Big( \sum_{\substack{i_1, \ldots, i_k \\ \mathrm{distinct}}} \alpha_{i_1}^1 \cdots \alpha_{i_k}^k \sum_{\substack{I \in \mathcal{I}_{n,d-k} \\ i_1, \ldots, i_k \notin I}} X_I t_{I \cup \{i_1, \ldots, i_k\} } \Big) \\
	&= \sup_{t \in \mathcal{T}} \sup_{\substack{\alpha^1, \ldots, \alpha^k \in \IR^n \\ \abs{\alpha^i} \le 1}} \Big\lVert\sum_{\substack{i_1, \ldots, i_{k} \\ \mathrm{distinct}}} \alpha_{i_1}^{1} \cdots \alpha_{i_k}^k \sum_{\substack{I \in \mathcal{I}_{n,d-k} \\ i_1, \ldots, i_k \notin I}} X_I t_{I \cup \{i_1,\ldots, i_k\}} \Big\rVert,
\end{split}
\end{align}
where for $k = d$ we use the convention $\mathcal{I}_{n,0} = \{ \emptyset \}$ and $X_{\emptyset} \coloneqq 1$.

One can interpret the quantities $W_k$ as follows: If $f_t(x) = \sum_{I \in \mathcal{I}_{n,d}} x_I t_I$ is the corresponding polynomial in $n$ variables, and $\nabla^{(k)} f_t(x)$ is the $k$-tensor of all partial derivatives of order $k$, then $W_k = \sup_{t \in \mathcal{T}} \abs{\nabla^{(k)} f_t(X)}_{\mathrm{op}}$. In this sense, we are considering the same quantities as in Theorem \ref{theorem:SupremaOfFunctions} but replace the difference operator $\mathfrak{h}$ by formal derivatives of the polynomial under consideration.

Furthermore, the concentration inequalities are phrased with the help of the quantities
\begin{align*}
	\tilde{W}_k &\coloneqq \sup_{\substack{\alpha^1, \ldots, \alpha^k \in \IR^n \\ \abs{\alpha^i} \le 1}} \sum_{\substack{i_1, \ldots, i_{k} \\ \mathrm{distinct}}} \alpha_{i_1}^1 \cdots \alpha_{i_k}^k \sup_{t \in \mathcal{T}}  \sup_{v^* \in \mathcal{B}_1^*} v^*\Big(\sum_{\substack{I \in \mathcal{I}_{n,d-k} \\ i_1, \ldots, i_k \notin I}} X_I t_{I \cup \{i_1, \ldots, i_k\} } \Big) \\
	&= \sup_{\substack{\alpha^1, \ldots, \alpha^k \in \IR^n \\ \abs{\alpha^i} \le 1}} \sum_{\substack{i_1, \ldots, i_{k} \\ \mathrm{distinct}}} \alpha_{i_1}^1 \cdots \alpha_{i_k}^k \sup_{t \in \mathcal{T}} \Big\lVert \sum_{\substack{I \in \mathcal{I}_{n,d-k} \\ i_1, \ldots, i_k \notin I}} X_I t_{I \cup \{i_1, \ldots, i_k\}} \Big\rVert.
\end{align*}
Clearly $\tilde{W}_k \ge W_k$ holds for all $k \in [d]$.

Concentration properties for functionals as in \eqref{eqn:fTauNew} have been studied for independent Rademacher variables $X_1, \ldots, X_n$ (i.\,e. $\IP(X_i = +1) = \IP(X_i = -1) = 1/2$) and $\mathcal{B} = \mathbb{R}$ in \cite[Theorem 14]{BBLM05} for all $d \ge 2$, and under certain technical assumptions in \cite{Ad15}. We prove deviation inequalities in the weakly dependent setting, and afterwards discuss how these compare to the particular result in \cite{BBLM05}. It is easily possible to derive a similar result for functions of independent random variables (in the spirit of Theorem \ref{theorem:Lpnormestimates}). As the corresponding proof is easily done by generalizing the proof of \cite[Theorem 14]{BBLM05}, we omit it.

\begin{theorem}			\label{theorem:ChaosInIndependentOrPartialLSI}
	Let $X = (X_1, \ldots, X_n)$ be a random vector in $\IR^n$ with support in $[a,b]^n$ satisfying a $\dpartial\mathrm{-LSI}(\sigma^2)$. For $f = f(X)$ as in \eqref{eqn:fTauNew} and all $p \ge 2$ we have
	\begin{align}		\label{eqn:LpEstimateNumber1}
		\norm{(f - \IE f)_+}_p &\le \sum_{j = 1}^d (2\sigma^2 (b-a)^2 (p-3/2))^{j/2} \IE W_j, \\
    \norm{f - \IE f}_p &\le \sum_{j = 1}^d (2(b-a)^2 p)^{j/2} \IE\tilde{W}_j. \label{eqn:LpEstimateNoPositivePart}
	\end{align}
	Consequently, for any $t \ge 0$
	\begin{equation}		\label{eqn:ConcentrationEstimateNumber1}
	\begin{split}
		\IP \left( f - \IE f \ge t \right) &\le 2 \exp \Big( - \frac{1}{2\sigma^2(b-a)^2} \min_{k=1,\ldots,d} \Big( \frac{t}{de \IE W_k} \Big)^{2/k} \Big) \\
		&\le 2 \exp \Big( -\frac{1}{2e^2\sigma^2(b-a)^2d^2} \min_{k=1,\ldots,d} \Big( \frac{t}{\IE W_k} \Big)^{2/k}  \Big)
	\end{split},
	\end{equation}
  and the same concentration inequalities hold with $\IE W_k$ replaced by $\IE \tilde{W}_k$.
\end{theorem}

Note that independent Rademacher random variables satisfy a $\dpartial\mathrm{-LSI}(1)$ (see e.\,g. \cite[Theorem 3]{Gr75} or \cite[Example 3.1]{DSC96}). Therefore, we get back \cite[Theorem 14]{BBLM05} from Theorem \ref{theorem:ChaosInIndependentOrPartialLSI} (with slightly different constants). However, Theorem \ref{theorem:ChaosInIndependentOrPartialLSI} moreover includes many models with dependencies like those discussed in the introduction. Therefore, it may be considered as a extension of \cite[Theorem 14]{BBLM05} to dependent situations and moreover to coefficients from any Banach space $\mathcal{B}$. For instance, we may consider an Ising chaos as a natural generalization of a Rademacher chaos to a dependent situation. In this case, Theorem \ref{theorem:ChaosInIndependentOrPartialLSI} yields that that we still obtain basically the same concentration properties if the dependencies are sufficiently weak (which is guaranteed by the conditions outlined in the introduction).

To illustrate our results further, let us consider the case of $d=2$ separately. Here we write
\begin{align*}
	T_1 &\coloneqq \IE W_1 = \IE \sup_{t \in \mathcal{T}} \sup_{v^* \in \mathcal{B}^*_1} \Big( \sum_{i = 1}^n \Big( \sum_{j = 1}^n X_j v^*(t_{ij}) \Big)^2 \Big)^{1/2} \\ 
	T_2 &\coloneqq \IE W_2 = \sup_{t \in \mathcal{T}} \sup_{v^* \in \mathcal{B}_1^*} \norm{(v^*(t_{ij}))_{i,j}}_{\mathrm{op}}.
\end{align*}
The following corollary follows directly from Theorem \ref{theorem:ChaosInIndependentOrPartialLSI}.

\begin{corollary}			\label{corollary:ChaosInBanachSpaces}
	Assume that $X = (X_1, \ldots, X_n)$ satisfies a $\dpartial\mathrm{-LSI}(\sigma^2)$ and is supported in $[a,b]^n$ and let $f_{\mathcal{T}} = f_{\mathcal{T}}(X)$ be as in \eqref{eqn:fTauNew} with $d = 2$. We have for all $t \ge 0$
	\begin{align*}
		\IP\left( f_{\mathcal{T}}(X) - \IE f_{\mathcal{T}}(X) \ge t \right) \le 2 \exp \Big( - \frac{1}{60(b-a)^2 \sigma^2} \min \Big( \frac{t^2}{T_1^2}, \frac{t}{T_2} \Big)\Big).
	\end{align*}
\end{corollary}

For the case of independent Rademacher variables, this recovers the upper tail in a famous result by Talagrand \cite[Theorem 1.2]{Tal96a} on concentration properties of quadratic forms in Banach spaces, which has also been done in \cite{BBLM05}.
Note that for $\mathcal{B} = \mathbb{R}$, we have
\begin{equation*}
	T_1 = \IE \sup_{t \in \mathcal{T}} \Big( \sum_{i = 1}^n \Big( \sum_{j = 1}^n t_{ij} X_j \Big)^2 \Big)^{1/2},\qquad
	T_2 = \sup_{t \in \mathcal{T}} \abs{T}_{\mathrm{op}},
\end{equation*}
where $T$ is the symmetric matrix with zero diagonal and entries $T_{ij} = t_{ij}$ if $i<j$. If $\mathcal{T}$ consists of a single element only, we have $T_1 \le \abs{T}_\mathrm{HS}$. Hence, Corollary \ref{corollary:ChaosInBanachSpaces} can be regarded as a generalized Hanson--Wright inequality.

\subsection{The Boolean hypercube}
The case of independent Rademacher random variables above can be interpreted in terms of quantities from Boolean analysis. Recall that any function $f: \{-1,+1\}^n \to \IR$ can be decomposed using the orthonormal \emph{Fourier--Walsh basis} given by $(x_S)_{S \subseteq [n]}$ for $x_S \coloneqq \prod_{i \in S} x_i$. More precisely, we have
\[
f(x) = \sum_{S \subset [n]} \hat{f}_S x_S = \sum_{j \in [n]} \sum_{S \subseteq [n] : \abs{S} = j} \hat{f}_S x_S,
\]
where the $(\hat{f}_S)_{S \subset [n]}$ are given by $\hat{f}_S = \int x_S f d\mu$ and are called the \emph{Fourier coefficients} of $f$. For any $j \in [n]$ we define the \emph{Fourier weight of order $j$} as $W_j(f) \coloneqq \sum_{S \subseteq [n] : \abs{S} = j} \hat{f}_S^2$. It is clear that $\norm{f}_2^2 = \sum_{j = 0}^n W_j(f)$. The following multilevel concentration inequality can now be easily deduced.

\begin{proposition}\label{proposition:RademacherAsHypercube}
Let $X_1, \ldots, X_n$ be independent Rademacher random variables and let $f: \{1,+1\}^n \to \IR$ be a function given in the Fourier--Walsh basis as $f(x) = \sum_{j = 0}^d \hat{f}_S x_S$ for some $d \in \IN, d \le n$. For any $t > 0$ we have
\[
\IP(\abs{f(X) - \IE f(X)} \ge t) \le \exp\Big(1 - \min_{j = 1,\ldots,d} \Big( \frac{t}{de W_j(f)^{1/2}} \Big)^{2/j} \Big).
\]
In other words, the event $\abs{f(X) - \IE f(X)} \le de \max_{j =1,\ldots, d} (W_j(f) t^j)^{1/2}$ holds with probability at least $1 - \exp(1-t)$.
\end{proposition}

The literature on Boolean functions is vast, and a modern overview is given in \cite{OD14}. Especially for concentration results we may highlight \cite[Theorem 1.4]{AW15} (which in particular holds for Boolean functions), which we discuss further and partially generalize to dependent models in Section \ref{subsection:Polyn+subgrc}. Proposition \ref{proposition:RademacherAsHypercube} may be of interest due to the direct use of quantities from Fourier analysis. Finally, we should add that while many concentration results for Boolean functions like \cite[Theorem 1.4]{AW15} or also Proposition \ref{proposition:RademacherAsHypercube} are valid for functions whose Fourier--Walsh decomposition stops at some order $d$, Theorem \ref{theorem:Lpnormestimates} or Theorem \ref{theorem:partialLSITails} work for functions with Fourier--Walsh decomposition possibly up to order $n$.

\subsection{Concentration properties of $U$-statistics}
Another application of Theorems \ref{theorem:Lpnormestimates} and \ref{theorem:partialLSITails} are concentration properties of so-called \emph{$U$-statistics} which frequently arise in statistical theory. We refer to \cite{PG99} for an excellent monograph. More recently, concentration inequalities for $U$-statistics have been considered in \cite{Ad06}, \cite[Section 3.1.2]{AW15} and \cite[Corollary 1.3]{BGS18}.

Let $\mathcal{Y} = \mathcal{X}^n$ and assume that $X_1, \ldots, X_n$ are either independent random variables, or the vector $X = (X_1, \ldots, X_n)$ satisfies a $\dpartial\mathrm{-LSI}(\sigma^2)$. Let $h: \mathcal{X}^d \to \IR$ be a measurable, symmetric function with $h(X_{i_1}, \ldots, X_{i_d}) \in L^\infty(\IP)$ for any $i_1, \ldots, i_d$, and define $B \coloneqq \max_{i_1 \neq \ldots \neq i_d} \norm{h(X_{i_1}, \ldots, X_{i_d})}_{L^\infty(\IP)}$. We are interested in the concentration properties of the $U$-statistic with kernel $h$, i.\,e. of
\begin{align}\label{eqn:UStatisticKernelh}
	f(X) = \sum_{i_1 \neq \ldots \neq i_d} h(X_{i_1}, \ldots, X_{i_d}).
\end{align}

\begin{proposition}\label{proposition:UStatistics}
Let $X = (X_1, \ldots, X_n)$ be as above and $f = f(X)$ be as in \eqref{eqn:UStatisticKernelh}. There exists a constant $C > 0$ (the same as in Theorems \ref{theorem:Lpnormestimates} and \ref{theorem:partialLSITails}) such that for any $t \ge 0$
\[
\IP \Big( \abs{f - \IE f} \ge Bt \Big) \le 2 \exp\Big( -\frac{1}{C} \min_{k=1,\ldots, d} \Big( \frac{t}{\binom{d}{k}2^k n^{d-k/2}} \Big)^{2/k} \Big)
\]
and for some $C = C(d)$
\begin{equation}	\label{eqn:normalizedUStatistic}
\IP(n^{1/2-d} \abs{f - \IE f} \ge B t) \le 2 \exp\Big( - \frac{1}{4C} \min \Big( t^2, n^{1-1/d} t^{2/d} \Big) \Big).
\end{equation}
\end{proposition}

The normalization $n^{1/2 - d}$ in \eqref{eqn:normalizedUStatistic} is of the right order for $U$-statistics generated by a non-degenerate kernel $h$, i.\,e. $\mathrm{Var}(\IE_{X_1} h(X_1, \ldots, X_d)) > 0$, see \cite[Remarks 4.2.5]{PG99}. In the case of i.i.d. random variables $X_1, \ldots, X_n$ it states that 
\[
\frac{1}{n^{d-1/2}} \sum_{i_1 < \ldots < i_d} h(X_{i_1}, \ldots, X_{i_d}) \Rightarrow \mathcal{N}(0, d^2 \Var(\IE_{X_1} h(X_1, \ldots, X_d)))
\] 
whenever $\IE h(X_1, \ldots, X_d)^2 < \infty$. Actually, \eqref{eqn:normalizedUStatistic} shows that for $t \le n^{1/2}$ we have sub-Gaussian tails for any finite $n \in \IN$ for bounded kernels $h$.

Proposition \ref{proposition:UStatistics} improves upon our old result \cite[Corollary 1.3]{BGS18} by providing multilevel tail bounds, thus yielding much finer estimates than the exponential moment bound given in the earlier paper. Moreover, it does not only address independent random variables but also weakly dependent models. As compared to the results from \cite{Ad06} and \cite[Section 3.1.2]{AW15}, Proposition \ref{proposition:UStatistics} covers different types of measures, since in \cite{Ad06} independent random variables were considered, while in \cite{AW15} a Sobolev-type inequality was required, which does not include the various discrete models for which a $\mathfrak{d}$--LSI holds.

\subsection{Polynomials and subgraph counts in exponential random graph models}\label{subsection:Polyn+subgrc}
Lastly, let us once again consider polynomial functions. The case of independent random variables has been treated in \cite[Theorem 1.4]{AW15} under more general conditions, so we omit it and concentrate on weakly dependent random variables.

Let $f_d: \IR^n \to \IR$ be a multilinear (also called tetrahedral) polynomial of degree $d$, i.\,e. of the form
\begin{align}			\label{eqn:Definitionfd}
	f_d(x) \coloneqq \sum_{k = 1}^d \sum_{1 \le i_1 \ne \ldots \ne i_k \le n} a^k_{i_1 \ldots i_k} x_{i_1} \cdots  x_{i_k}
\end{align}
for symmetric $k$-tensors $a^k$ with vanishing diagonal. Here, a $k$-tensor $a^k$ is called symmetric, if $a^k_{i_1 \ldots i_k} = a^k_{\sigma(i_1) \ldots \sigma(i_k)}$ for any permutation $\sigma \in \mathcal{S}_k$, and the (generalized) diagonal is defined as $\Delta_k \coloneqq \{ (i_1, \ldots, i_k) : \abs{\{i_1, \ldots, i_k\}} <k \}$. Denote by $\nabla^{(k)} f$ the $k$-tensor of all partial derivatives of order $k$ of $f$.

For the next result, given some $d \in \IN$, we recall a family of norms $\norm{\cdot}_{\mathcal{I}}$ on the space of $d$-tensors for each partition $\mathcal{I} = \{ I_1, \ldots, I_k \}$ of $\{1,\ldots,d\}$. The family $\norm{\cdot}_{\mathcal{I}}$ has been first introduced in \cite{La06}, where it was used to prove two-sided estimates for $L^p$ norms of Gaussian chaos, and the definitions given below agree with the ones from \cite{La06} as well as \cite{AW15} and \cite{AKPS18}. For brevity, write $P_d$ for the set of all partitions of $\{1,\ldots, d\}$. For each $l = 1,\ldots, k$ we denote by $x^{(l)}$ a vector in $\IR^{n^{I_l}}$, and for a $d$-tensor $A = (a_{i_1, \ldots, i_d})$ set
\[
\norm{A}_{\mathcal{I}} \coloneqq \sup \Big\lbrace \sum_{i_1 \ldots i_d} a_{i_1 \ldots i_d} \prod_{l = 1}^k x^{(l)}_{i_{I_l}} : \sum_{i_{I_l}} (x^{(l)}_{i_{I_l}})^2 \le 1 \text{ for all } l = 1, \ldots, k\Big\rbrace.
\]
We can regard the $\norm{A}_{\mathcal{I}}$ as a family of operator-type norms. In particular, it is easy to see that $\norm{A}_{\{1, \ldots, d\}} = \abs{A}_\mathrm{HS}$ and $\norm{A}_{\{\{1\}, \ldots, \{d\}\}} = \abs{A}_\mathrm{op}$.

The following result has been proven in the context of Ising models (in the Dobrushin uniqueness regime) in \cite{AKPS18}, and can easily be extended to any vector $X$ satisfying a $\dpartial\mathrm{-LSI}(\sigma^2)$. By invoking the family of norms $\norm{\cdot}_{\mathcal{I}}$, it provides a refinement of our general result for the special case of multilinear polynomials.

\begin{theorem}		\label{theorem:concentrationHLSI}
	Let $X$ be a random vector supported in $[-1,+1]^n$ and satisfying a $\dpartial\mathrm{-LSI}(\sigma^2)$, and $f_d = f_d(X)$ be as in \eqref{eqn:Definitionfd}. There exists a constant $C> 0$ depending on $d$ only such that for all $t \ge 0$
	\begin{equation}	\label{eqn:partialLSIpolynomial}
		\IP\left( \abs*{f_d - \IE f_d} \ge t\right) \le 2 \exp \Big( - \frac{1}{C} \min_{k = 1,\ldots,d} \min_{\mathcal{I} \in P_k} \left( \frac{t}{\sigma^{k} \norm{\IE \nabla^{(k)} f_d}_{\mathcal{I}}} \right)^{2/\abs{\mathcal{I}}} \Big).
	\end{equation}
\end{theorem}

For illustration, let us once again consider the case of $d=2$. In the notation of \eqref{eqn:Definitionfd}, we take $a^1 = 0$ and $a^2 = A$, i.\,e. $f_2(x) = x^T A x$ for a symmetric matrix $A$ with vanishing diagonal. In this case, assuming the components of $X$ to be centered (so the the $k=1$ term vanishes), Theorem \ref{theorem:concentrationHLSI} reads
	\[
		\IP\left( \abs*{f_2 - \IE f_2} \ge t\right) \le 2 \exp \Big( -\frac{1}{C} \min \Big(\frac{t^2}{\sigma^4\abs{A}_\mathrm{HS}^2}, \frac{t}{\sigma^2\abs{A}_\mathrm{op}} \Big) \Big),
	\]
i.\,e. we obtain a Hanson--Wright inequality in this situation. For higher orders, we arrive at similar bounds. Altogether, for the class of multilinear polynomials, Theorem \ref{theorem:concentrationHLSI} yields finer bounds than Theorem \ref{theorem:partialLSITails} (by virtue of the large class of norms involved), though for $d \ge 3$ explicit calculations of the norms involved can be difficult.

To point out one possible application, Theorem \ref{theorem:concentrationHLSI} can be used in the context of the exponential random graph model (ERGM). Let us briefly recall the definitions. Given $s \in \IN$ real numbers $\beta_1, \ldots, \beta_s$ and simple graphs $G_1, \ldots, G_s$ (with $G_1$ being a single edge by convention), the ERGM with parameter $\mathbf{\beta} = (\beta_1, \ldots, \beta_s, G_1, \ldots, G_s)$ is a probability measure on the space of all graphs on $n \in \IN$ vertices given by the weight function $\exp\left( \sum_{i = 1}^s \beta_i n^{-\abs{V_i}+2} N_{G_i}(x) \right)$, where $N_{G_i}(x)$ is the number of copies of $G_i$ in the graph $x$ and $\abs{V_i}$ is the number of vertices of $G_i = (V_i, E_i)$. For details, see \cite{CD13} or \cite{SS18}. One can think of the ERGM as an extension of the famous Erd\"{o}s--R\'{e}nyi model (which corresponds to the choice $s=1$) to account for dependencies between the edges.
 
By way of example we show concentration properties of the number of triangles $T_3(X) = \sum_{\{e,f,g\} \in \mathcal{T}_3} X_e X_f X_g$ (where $\mathcal{T}_3$ denotes the set of all three edges forming a triangle). To formulate our results, we need to recall the function $\Phi_{\bb}(x) = \sum_{i=1}^s \beta_i \abs{E_i}x^{\abs{E_i}-1}$ which frequently appears in the discussion of the ERGM. Moreover, we set $\abs{\bb} \coloneqq (\abs{\beta_1}, \ldots, \abs{\beta_s})$. In the following corollary, the condition $\frac{1}{2} \Phi_{\abs{\bb}}'(1) < 1$ ensures weak dependence in the sense that a $\dpartial$--LSI holds. As outlined above, in comparison to earlier results like \cite[Theorem 3.2]{SS18}, using Theorem \ref{theorem:concentrationHLSI} yields sharper tail estimates.

\begin{corollary}		\label{corollary:ERGMTriangle}
	Let $X$ be an exponential random graph model with parameter $\bb =(\beta_1, \ldots, \beta_s, G_1, \ldots, G_s)$ such that $\frac{1}{2} \Phi_{\abs{\bb}}'(1) < 1$. There is a constant $C(\bb)$ such that for all $t \ge 0$
	\begin{align*}
		&\IP\left( \abs{T_3 - \IE T_3} \ge t \right) \\
		&\le 2 \exp \left( - \frac{1}{C(\bb)} \min \left( \frac{t^2}{\max(C_{S_2} n^4,C_E n^{3}, n^{3})}, \frac{t}{\max(\sqrt{2n},2C_E n)}, \frac{t^{2/3}}{2} \right) \right).
	\end{align*}
\end{corollary}

\section{Concentration inequalities under logarithmic Sobolev inequalities: Proofs}			\label{section:concentrationIneqUnderLSI}
In this section, we give the proofs of our main results. All of them work by first establishing a growth rate on the $L^p$ norms of $f - \IE f$ which will then be iterated. For technical reasons, we need to introduce some auxiliary difference operators which are closely related to $\mathfrak{h}$. For $i \in [n]$ let
\begin{equation*} 
\mathfrak{h}^+_if(X) = \lVert (f(X) - T_if(X))_+ \rVert_{X_i', \infty},
\qquad \mathfrak{h}^+f = (\mathfrak{h}^+_1f, \ldots, \mathfrak{h}^+_nf),
\end{equation*}
\begin{equation*}
\mathfrak{h}^-_if(X) = \lVert (f(X) - T_if(X))_- \rVert_{X_i', \infty},
\qquad \mathfrak{h}^-f = (\mathfrak{h}^-_1f, \ldots, \mathfrak{h}^-_nf),
\end{equation*}
where $\norm{f}_{X_i',\infty}$ shall denote the $L^\infty$ norm with respect to $X_i'$.

The $L^p$ norm inequalities which form the core of our proofs can be found in \cite[Theorem 2.3, Corollary 2.6]{BGS18} (building upon the earlier results in \cite{BBLM05}). Note that as compared to \cite{BGS18}, a different choice of normalization for $\mathfrak{h}^\pm$ leads to slightly different constants. 

\begin{theorem}				\label{theorem:BBLM}
If $X_1, \ldots, X_n$ are independent random variables and $f = f(X) \in L^\infty(\IP)$, with the constant $\kappa = \frac{\sqrt{e}}{2\,(\sqrt{e} - 1)}$, we have for any $p \ge 2$,
\begin{equation*} 
\lVert (f - \mathbb{E}f)_+ \rVert_p \le (2\kappa p)^{1/2}\, \lVert \mathfrak{h}^+f \rVert_p \quad \text{and} \quad
\lVert (f - \mathbb{E}f)_- \rVert_p \le (2\kappa p)^{1/2}\, \lVert \mathfrak{h}^-f \rVert_p.
\end{equation*}
Consequently, this leads to 
\[
	\norm{f - \IE f}_p \le (8\kappa p)^{1/2} \norm{\mathfrak{h}f}_p.
\]
\end{theorem} 

Furthermore, we need an auxiliary statement relating differences of consecutive order. In \cite{BGS18}, we have proven that $\abs{\mathfrak{h}\abs{\mathfrak{h}^{(d)}f}_\mathrm{HS}} \le \abs{\mathfrak{h}^{(d+1)}f}_\mathrm{HS}$. Moreover, we explained that a similar estimate with the Hilbert--Schmidt replaced by operator norms cannot be true. As we will see next, the key step in order to be able to invoke operator norms nevertheless is to work with $\mathfrak{h}^+$.

Here we need the following simple but crucial observation: if $A$ is a $d$-tensor, the supremum in the definition of $\abs{A}_\mathrm{op}$ is attained, and if $A$ is a non-negative tensor (i.\,e. $A_{i_1 \ldots i_d} \ge 0$ for all $i_1, \ldots, i_d$), the maximizing vectors $\tilde{v}^1, \ldots, \tilde{v}^d$ can be chosen to have all positive entries. Indeed, since $\tilde{v}^1_{i_1} \cdots \tilde{v}^d_{i_d} \le \abs{\tilde{v}^1_{i_1} \cdots \tilde{v}^d_{i_d}}$, we can define $\abs{\tilde{v}}^j$ by taking the absolute value element-wise.

\begin{lemma}			\label{lemma:recursiveEstimateOperatorNorms}
For any $d \ge 2$
\[
	\abs{\mathfrak{h}^+ \abs{\mathfrak{h}^{(d-1)} f(X)}_\mathrm{op}} \le \abs{\mathfrak{h}^{(d)} f(X)}_\mathrm{op}.
\]
\end{lemma}

\begin{proof}
We have
\begin{align*}
	&\abs{\mathfrak{h}^+ \abs{\mathfrak{h}^{(d-1)} f}_\mathrm{op}}^2 = \sum_{i= 1}^n \norm*{\left( \abs{\mathfrak{h}^{(d-1)} f}_\mathrm{op} - \abs{\mathfrak{h}^{(d-1)} T_i f}_\mathrm{op} \right)_+}_{i,\infty}^2 \\
	&= \sum_{i=1}^n \bnorm{\left( \sup_{v^j} \skal{v^1 \cdots v^{d-1}, \mathfrak{h}^{(d-1)}f} - \sup_{v^j} \skal{v^1 \cdots v^{d-1}, \mathfrak{h}^{(d-1)}T_i f} \right)_+}_{i,\infty}^2 \\
 	&\le \sum_{i = 1}^n \bnorm{\bigg(\skal{\tilde{v}^1\cdots \tilde{v}^{d-1}, \mathfrak{h}^{(d-1)}f - \mathfrak{h}^{(d-1)} T_i f} \bigg)_+}_{i,\infty}^2 \\
 	&\le \sum_{i = 1}^n \bnorm{\sum_{i_1,\ldots,i_d-1} \tilde{v}^1_{i_1} \cdots \tilde{v}^{d-1}_{i_{d-1}} \bnorm{(\Id - T_i) \prod_{j = 1}^{d-1} (\Id - T_{i_s})f}_{i_1 \cdots i_{d-1},\infty} }_{i,\infty}^2 \\
 	&\le \sum_{i =1}^n \bigg( \sum_{i_1, \ldots, i_{d-1}} \tilde{v}^1_{i_1} \cdots \tilde{v}^{d-1}_{i_1} \mathfrak{h}_{i i_1 \cdots i_{d-1}}f \bigg)^2 \\
 	&= \bigg( \sup_{v^d : \abs{v^d} \le 1} \sum_{i_d=1}^n \sum_{i_1, \ldots, i_{d-1}} \tilde{v}_{i_1}^1 \cdots \tilde{v}^{d-1}_{i_{d-1}} v^{d}_{i_d} \mathfrak{h}_{i_1 \cdots i_d} f \bigg)^2 \\
 	&\le \bigg( \sup_{v^1, \ldots, v^d : \abs{v^j} \le 1} \sum_{i_1, \ldots, i_d} v_{i_1}^1 \cdots v_{i_d}^d \mathfrak{h}_{i_1 \cdots i_d}f \bigg)^2 \\
 	&= \abs{\mathfrak{h}^{(d)}f}_\mathrm{op}^2
\end{align*}
Here, in the first inequality we insert the vectors $\tilde{v}^1, \ldots, \tilde{v}^{d-1}$ maximizing the supremum and use the monotonicity of $x \mapsto x_+$, and the second and third inequality follow from the triangle inequality. Taking the square root yields the claim.
\end{proof}

As a final step, we need to establish a connection between $L^p$ norm estimates and multilevel concentration inequalities. This is given by the following proposition, which was proven in \cite[Theorem 7]{Ad06} and \cite[Theorem 3.3]{AW15}. We state it in the form given in \cite[Proof of Theorem 3.6]{SS18} with slight modifications.

\begin{proposition}			\label{proposition:FromLpToCoM}
	Assume that a random variable $f$ satisfies for any $p \ge 2$ and some constants $C_1, \ldots, C_d \ge 0$ $\norm{f - \IE f}_p \le \sum_{k = 1}^{d} C_k (p-s)^{k/2}$ for some $s \in [0,2)$,
	and let $L \coloneqq \abs{\{ l : C_l > 0\}}$. For any $t \ge 0$ we have
	\begin{align*}
			\IP(\abs{f - \IE f} \ge t) \le 2 \exp \Big( - \min\Big(\frac{\log(2)}{2-s},1\Big) \min_{k = 1, \ldots, d} \Big( \frac{t}{Le C_k} \Big)^{2/k} \Big).
	\end{align*}
\end{proposition}

We will not give a proof of Proposition \ref{proposition:FromLpToCoM} and refer to the aforementioned works. However, the proof is almost identical to the proof of Proposition \ref{proposition:RademacherAsHypercube}. The two important cases will be $s = 0$ (for independent random variables) as well as $s = 3/2$ (in the weakly dependent setting).


The proof of Theorem \ref{theorem:Lpnormestimates} is now easily completed.

\begin{proof}[Proof of Theorem \ref{theorem:Lpnormestimates}]
Since $X_1, \ldots, X_n$ are independent, Theorem \ref{theorem:BBLM} yields
	\begin{align*}
		\norm{f - \IE f}_p \le (8\kappa p)^{1/2} \norm{\mathfrak{h}f}_p \le (8\kappa p)^{1/2} \norm{\mathfrak{h}f}_{\mathrm{op},1} + (8\kappa p)^{1/2} \norm{(\abs{\mathfrak{h} f} - \IE \abs{\mathfrak{h} f})_+}_p
	\end{align*}
	where we have used that for any positive random variable $W$
	\begin{align}		\label{eqn:LpEstimatePositiveFunction}
		\norm{W}_p \le \IE W + \norm{(W - \IE W)_+}_p.
	\end{align}
	The second term on the right hand side can now be estimated using Theorem \ref{theorem:BBLM} again, which in combination with Lemma \ref{lemma:recursiveEstimateOperatorNorms} gives
	\begin{align*}
		\norm{\left( \abs{\mathfrak{h} f} - \IE \abs{\mathfrak{h}f} \right)_+}_p \le \sqrt{2 \kappa p} \norm{\mathfrak{h}^+ \abs{\mathfrak{h}f}}_p \le \sqrt{2\kappa p} \norm{\mathfrak{h}^{(2)}f}_{\mathrm{op},p}.
	\end{align*}
	This can be easily iterated to obtain for any $d \in \IN$
	\begin{align*}
	\norm{f - \IE f}_p \le \sum_{j = 1}^{d-1} (8\kappa p)^{j/2} \norm{\mathfrak{h}^{(j)}f}_{\mathrm{op},1} + (8\kappa p)^{d/2} \norm{\mathfrak{h}^{(d)}f}_{\mathrm{op},\infty}.
	\end{align*}
Now it remains to apply Proposition \ref{proposition:FromLpToCoM}.
\end{proof}

To prove Theorem \ref{theorem:partialLSITails}, we shall require the following proposition, which is proven in \cite[Proposition 2.4]{GSS18}. (Note that the definition of $\mathfrak{h}$ there differed by a factor of $\sqrt{2}$.) The estimate \eqref{eqn:fMinusEfLP2} does not appear therein, but is an easy modification of the proof.

\begin{proposition}	\label{proposition:momentinequality}
	Let $\mu$ be a measure on a product of Polish spaces satisfying a $\dpartial\mathrm{-LSI}(\sigma^2)$. Then, for any $f \in L^\infty(\mu)$ and any $p \ge 2$ we have
  \begin{equation} \label{eqn:fMinusEfLP}
	\norm{f - \IE f}_p \le (2\sigma^2 (p-3/2))^{1/2} \norm{\dpartial f}_p \le (\sigma^2(p-3/2)/2)^{1/2} \norm{\mathfrak{h} f}_p
  \end{equation}
  and
  \begin{equation}\label{eqn:fMinusEfLP2}
	\norm{(f - \IE f)_+}_p \le (2\sigma^2 (p-3/2))^{1/2} \norm{\mathfrak{h}^+ f}_p.
  \end{equation}
\end{proposition}

\begin{proof}[Proof of Theorem \ref{theorem:partialLSITails}]
The proof is very similar to the proof of Theorem \ref{theorem:Lpnormestimates}.
In the first step, using \eqref{eqn:fMinusEfLP} leads to
\[
	\norm{f - \IE f}_p \le (2\sigma^2 (p-3/2))^{1/2} \norm{\mathfrak{h}f}_{\mathrm{op},1} + (2\sigma^2 (p-3/2))^{1/2} \norm{(\abs{\mathfrak{h}f} - \IE \abs{\mathfrak{h} f})_+}_p.
\]
Equation \eqref{eqn:fMinusEfLP2} can be used to estimate the second term on the right hand side. So, for any $d \in \IN$ we have by an iteration
\[
	\norm{f - \IE f}_p \le \sum_{j = 1}^{d-1} (2\sigma^2(p-3/2))^{j/2} \norm{\mathfrak{h}^{(j)} f}_{\mathrm{op},1} + (2\sigma^2(p-3/2))^{d/2} \norm{\mathfrak{h}^{(d)} f}_{\mathrm{op},\infty}
\]
Again we can apply Proposition \ref{proposition:FromLpToCoM} to obtain the concentration inequality.
\end{proof}

To prove Theorem \ref{theorem:SupremaOfFunctions} we shall need the following lemma.

\begin{lemma}\label{lemma:hPlusEstimateSupremum}
Let $(\mathcal{B}, \norm{\cdot})$ be a Banach space and $\mathcal{F}$ a family of uniformly norm-bounded, $\mathcal{B}$-valued, measurable functions and set $g(X) = \sup_{f \in \mathcal{F}} \norm{f(X)}$. We have
\[
	\abs{\mathfrak{h}^+ g(X)} \le \sup_{f \in \mathcal{F}} \abs{\mathfrak{h}^+ \norm{f}(X)}.
\]
\end{lemma}

\begin{proof}
 Fix an $X \in \mathcal{Y}$ and choose for any $\epsilon > 0$ a function $f_\epsilon$ such that $\norm{f_\epsilon(X)} \ge \sup_{f \in \mathcal{F}} \norm{f(X)} - \epsilon$. This yields
{\footnotesize
\begin{align*}
	\mathfrak{h}^+_i g(X) &= \sup_{x_i'} \Big( \sup_{f \in \mathcal{F}} \norm{f(X)} - \sup_{f \in \mathcal{F}} \norm{f(X_{i^c}, x_i')}\Big)_+ 
	\le \sup_{x_i'} \Big( \norm{f_\epsilon(X)} + \epsilon - \norm{f_\epsilon(X_{i^c}, x_i')} \Big)_+ 
	\\ &\le \sup_{x_i'} (\norm{f_\epsilon(X)} - \norm{f_\epsilon(X_{i^c}, x_i') })_+ + \epsilon = \mathfrak{h}_i^+ \norm{f_\epsilon}(X) + \epsilon,
\end{align*}
}where the first inequality follows by monotonicity of $x \mapsto x_+$ and the second one is a consequence of $(a+b-c)_+ \le (a-c)_+ + b$ for $a,b,c \ge 0$. Thus we have
{\footnotesize
\begin{align*}
	\abs{\mathfrak{h}^+ g(X)} &= \Big(\sum_{i = 1}^n \mathfrak{h}_i^+ g(X)^2 \Big)^{1/2} \le \Big( \sum_{i = 1}^n (\mathfrak{h}_i^+ \norm{f_\epsilon}(X) + \epsilon)^2 \Big)^{1/2} = \abs{\mathfrak{h}^+ \norm{f_\epsilon}(X) + \epsilon(1,\ldots,1)} \\
	&\le \abs{\mathfrak{h}^+ \norm{f_\epsilon}(X)} + \sqrt{n} \epsilon \le \sup_{f \in \mathcal{F}} \abs{\mathfrak{h}^+ \norm{f}(X)} + \sqrt{n} \epsilon.
\end{align*}
}Taking the limit $\epsilon \to 0$ yields the claim.
\end{proof}

\begin{proof}[Proof of Theorem \ref{theorem:SupremaOfFunctions}]
Note that in the real-valued case, the estimate $\mathfrak{h}_i^+ \abs{f} \le \mathfrak{h}_i f$ holds.  For brevity, let $s = 3/2$. Using this in combination with Proposition \ref{proposition:momentinequality} and Lemma \ref{lemma:hPlusEstimateSupremum} yields
\begin{align*}
\norm{(g - \IE g)_+}_p &\le (2\sigma^2 (p-s))^{1/2} \norm{\mathfrak{h}^+ g}_p 
\le (2\sigma^2 (p-s))^{1/2} \norm{\sup_{f \in \mathcal{F}} \abs{\mathfrak{h} f}}_p \\
&\le (2\sigma^2 (p-s))^{1/2} \IE W_1 + (2\sigma^2 (p-s))^{1/2} \norm{(W_1 - \IE W_1)_+}_p.
\end{align*}
We can apply Proposition \ref{proposition:momentinequality} again on the right hand side, which gives
\[
	\norm{(g - \IE g)_+}_p \le (2\sigma^2 (p-s))^{1/2} \IE W_1 + (2\sigma^2 (p-s)) \norm{\mathfrak{h}^+ W_1}_p.
\]
A combination of Lemmas \ref{lemma:recursiveEstimateOperatorNorms} and \ref{lemma:hPlusEstimateSupremum} shows that $\abs{\mathfrak{h}^+ W_j} \le W_{j+1}$, and so by an iteration we obtain
\[
    \norm{(g - \IE g)_+}_p \le \sum_{j = 1}^{d-1} (2\sigma^2 (p-s))^{j/2} \IE W_j + (2\sigma^2(p-s))^{d/2} \norm{W_d}_\infty.the
\]
In the case of independent random variables we replace the first step using Theorem \ref{theorem:BBLM}. Here, $2\sigma^2 = 2\kappa$ and $s = 0$.
\end{proof} 

\begin{proof}[Proof of Proposition \ref{prop:SupremaOfSums}]
The proof shares some similarities with the proof of Lemma \ref{lemma:hPlusEstimateSupremum}. Since $X$ satisfies a $\dpartial\mathrm{-LSI}(\sigma^2)$, we have for any $p \ge 2$
	\[
		\norm{(g - \IE g)_+}_p \le (2\sigma^2 (p-3/2))^{1/2} \norm{\mathfrak{h}^+ g}_p.
	\]
	Moreover, for any $i \in [n]$ and $x \in \mathcal{Y}$, if a maximizer $\tilde{f}$ of $\sup_{f \in \mathcal{F}} \abs{\sum_{j = 1}^n f(x_j)}$ exists, we obtain
	\begin{align*}
		\mathfrak{h}^+_i g(x)^2 &= \sup_{x_i'} \Big(\sup_{f \in \mathcal{F}} \abs{f(X)} - \sup_{f \in \mathcal{F}} \abs{f(X_{i^c}, x_i')}\Big)^2_+ \\
		&\le \sup_{x_i'} \Big( \abs{\tilde{f}(X)} - \abs{\tilde{f}(X_{i^c}, x_i')} \Big)_+^2 \le c(\tilde{f})^2 \le \sup_{f \in \mathcal{F}} c(f)^2.
	\end{align*}
    If a maximizer $\tilde{f}$ does not exist, these estimates remain valid by an approximation argument as in the proof of Lemma \ref{lemma:hPlusEstimateSupremum}.
	Consequently, we have $\norm{(g - \IE g)_+}_p \le (2\sigma^2 (p-3/2) n \sup_{f \in \mathcal{F}} c(f)^2 )^{1/2}.$ The claim now follows from Proposition \ref{proposition:FromLpToCoM}.
\end{proof}

\section{Suprema of chaos, U-statistics and polynomials: Proofs}		\label{section:SupremaOfChaos}
\begin{proof}[Proof of Theorem \ref{theorem:ChaosInIndependentOrPartialLSI}]
Let us first consider the case that $X$ satisfies a $\dpartial\mathrm{-LSI}(\sigma^2)$. Recall that we have by \eqref{eqn:fMinusEfLP2}
\[
	\norm{(f - \IE f)_+}_p \le (2\sigma^2 (p-3/2))^{1/2} \norm{\mathfrak{h}^+ f}_p.
\]
We shall make use of the pointwise inequality $\abs{\mathfrak{h}^+ f} \le (b-a) W_1.$ To see this, let $(\tilde{t}, \tilde{v}^*)$ be the tuple satisfying $\sup_{t \in \mathcal{T}} \sup_{v^* \in \mathcal{B}_1^*} v^*(\sum_{I \in \mathcal{I}_{n,d}} X_I t_I) = \tilde{v}^*(\sum_{I \in \mathcal{I}_{n,d}} X_I \tilde{t}_I)$. We have
\begin{align*}
	\abs{\mathfrak{h}^+ f(X)}^2 &= \sum_{i = 1}^n \sup_{x_i'} \Big( {\sup_{t, v^*} v^*\Big(\sum_{I \in \mathcal{I}_{n,d}} X_I t_I \Big)} - {\sup_{t, v^*} v^*\Big(\sum_{I \in \mathcal{I}_{n,d}} (X_{i^c}, x_i')_I t_I \Big)} \Big)_+^2 \\
	&\le \sum_{i = 1}^n \sup_{x_i'} \Big( (X_i - x_i') \sum_{\substack{I \in \mathcal{I}_{n,d-1} \\ i \notin I}} \tilde{v}^*(X_I \tilde{t}_{I \cup \{i\}}) \Big)^2 \\
	&\le (b-a)^2 \sum_{i = 1}^n \Big( \tilde{v}^*\Big( \sum_{\substack{I \in \mathcal{I}_{n,d-1} \\ i \notin I}} X_I \tilde{t}_{I \cup \{i\}} \Big) \Big)^2 \\
	&= (b-a)^2 \sup_{\alpha^{(1)} : \norm{\alpha^{(1)}} \le 1} \Big(\tilde{v}^* \Big( \sum_{i = 1}^n \alpha^{(1)}_i \sum_{I \in \mathcal{I}_{n,d-1} : i \notin I} X_I \tilde{t}_{I \cup \{i\}} \Big)\Big)^2 \\
	&\le (b-a)^2 \Big(\sup_{t, v^*} \sup_{\alpha^{(1)} : \norm{\alpha^{(1)}} \le 1} v^*\Big( \sum_{i = 1}^n \alpha_i^{(1)} \sum_{I \in \mathcal{I}_{n,d} : i \notin I} X_I t_{I \cup \{i\}} \Big)\Big)^2 \\
	&= (b-a)^2 W_1^2,
\end{align*}
proving the first part. Consequently,
\begin{align*}
	\norm{(f - \IE f)_+}_p \le (2\sigma^2 (b-a)^2 (p-3/2))^{1/2}\left( \IE W_1 + \norm{(W_1 - \IE W_1)_+}_p \right). 
\end{align*}
As in \cite{BBLM05}, this can now be iterated, i.\,e. we have for any $k \in \{1,\ldots, d-1\}$ $\abs{\mathfrak{h}^+ W_k} \le (b-a)W_{k+1}$. Here we may argue as above, where the only difference is to choose $(\tilde{t}, \tilde{v}^*)$ and $\tilde{\alpha}^{(1)},\ldots, \tilde{\alpha}^{(k)}$ which maximize $W_k$. This finally leads to
\begin{align*}
	\norm{f - \IE f}_p \le \sum_{j = 1}^d (2\sigma^2 (b-a)^2 (p-3/2))^{j/2} \IE W_j,
\end{align*}
using that $W_d$ is constant. This proves \eqref{eqn:LpEstimateNumber1}. The same arguments are also valid without a $\dpartial\mathrm{-LSI}(\sigma^2)$ property, if one considers $\norm{(f - \IE f)_+}_p$ and applies Theorem \ref{theorem:BBLM} instead. \par
Lastly, to prove \eqref{eqn:LpEstimateNoPositivePart}, let us first consider why we cannot argue as before. Note that the argument heavily relies on the positive part of the difference operator $\mathfrak{h}^+$, which allows us to choose the maximizers $t_1, \ldots, t_n$ independent of $i \in [n]$. This is no longer possible for the concentration inequality. Here, Theorem \ref{theorem:BBLM} yields
\begin{align*}			
	\norm{f - \IE f}_p &\le (\sigma^2 p)^{1/2} \norm{\mathfrak{h}f}_{p} \\
	\norm{(f - \IE f)_+}_p &\le (\sigma^2 p)^{1/2} \norm{\mathfrak{h}^+ f}_{p}. 		
\end{align*}
Thus this argument fails if we try to use these inequalities. However, we can rewrite
$\mathfrak{h}_i f(x) = \sup_{x_i',x_i''} (f(x_{i^c}, x_i') - f(x_{i^c}, x_i''))_+= \sup_{x_i'} \mathfrak{h}_i^+ f(x_{i^c}, x_i')$, where the $\sup$ is to be understood with respect to the support of $X_i'$. As a consequence, we have for each fixed $i \in [n]$ (again choosing $\tilde{t}$ by maximizing the first summand in the brackets)
\begin{align*}
	\mathfrak{h}_i f(x)^2 &= \sup_{x_i'} \sup_{x_i''} \Big( \sup_{t \in \mathcal{T}} \bnorm{\sum_{I \in \mathcal{I}_{n,d}} (X_{i^c}, x_i')_I t_I} - \sup_{t \in \mathcal{T}} \bnorm{\sum_{I \in \mathcal{I}_{n,d}} (X_{i^c}, x_i'')_I t_I} \Big)_+^2 \\
	&\le \sup_{x_i'} \sup_{x_i''} \bnorm{(x_i' - x_i'')\sum_{I \in \mathcal{I}_{n,d-1} : i \notin I} X_I \tilde{t}_{I \cup \{i\}} }^2 \\
	&\le \sup_{x_i', x_i''} \abs{x_i' - x_i''}^2 \sup_{t \in \mathcal{T}} \bnorm{\sum_{I \in \mathcal{I}_{n,d-1} : i \notin I} X_I t_{I \cup \{i\}} }^2 \\
	&\le (b-a)^2 \sup_{t \in \mathcal{T}} \bnorm{\sum_{I \in \mathcal{I}_{n,d-1}: i \notin I} X_I t_{I \cup \{i\}} }^2.
\end{align*}	
This implies
\begin{equation*}
	\abs{\mathfrak{h}f}^2(x) \le (b-a)^2 \sup_{\substack{\alpha^1 \in \IR^n \\ \abs{\alpha^1} \le 1}} \sum_{i = 1}^n \alpha_i^1 \sup_{t \in \mathcal{T}} \bnorm{\sum_{\substack{I \in \mathcal{I}_{n,d-1} \\ i \notin I}} X_I t_{I \cup \{i \}} } = (b-a)^2 \tilde{W}_1^2.
\end{equation*}
The proof is now completed as using the same arguments as in the first part, with $W_k$ replaced by $\tilde{W}_k$. The same argument is valid for $X$ satisfying a $\dpartial\mathrm{-LSI}(\sigma^2)$.
\end{proof}


\begin{proof}[Proof of Proposition \ref{proposition:RademacherAsHypercube}]
The proposition can be proven using a similar technique as before, since the Hilbert--Schmidt norms of higher order difference act as Fourier projections. We choose to take an alternate route as follows. The proof of \cite[Theorem 9.21]{OD14} shows that for any $f$ with degree at most $d$ and any $p \ge 2$
\begin{equation} \label{eqn:LpInequalityBoole}
\norm{f(X) - \IE f(X)}_p \le \sum_{j = 1}^d (p-1)^{j/2} W_j(f)^{1/2}.
\end{equation}
First off, by Chebyshev's inequality we have for any $p \ge 1$
\[
\IP(\abs{f(X) - \IE f(X)} \ge e\norm{f(X) - \IE f(X)}_p) \le \exp(-p).
\]
We want to apply this to a $t$-dependent parameter $p$ given by the function
\[
\eta_f(t) \coloneqq 1 + \min_{j = 1,\ldots,d} \Big( \frac{t}{de W_j(f)^{1/2}} \Big)^{2/j}.
\]
If $\eta_f(t) \ge 2$, \eqref{eqn:LpInequalityBoole} yields $e\norm{f(X) - \IE f(X)}_{\eta_f(t)} \le t$, which combined with the trivial estimate $\IP(\cdot) \le 1$ gives
\[
\IP(\abs{f(X) - \IE f(X)} \ge t) \le e^2 \exp(-\eta_f(t)) = \exp\Big(1 - \min_{j = 1,\ldots,d} \Big( \frac{t}{de W_j(f)^{1/2}} \Big)^{2/j}\Big)
\]
as claimed.
\end{proof}

\begin{proof}[Proof of Proposition \ref{proposition:UStatistics}]
	We apply Theorems \ref{theorem:Lpnormestimates} and \ref{theorem:partialLSITails} in the respective cases. To this end, we make use of the general bound $\norm{\mathfrak{h}^{(k)} f}_{\mathrm{op},1} \le \norm{\mathfrak{h}^{(k)}f}_{\mathrm{HS},\infty}$ for $k \in [d]$. For any distinct $j_1, \ldots, j_k$ write $\norm{\cdot} = \norm{\cdot}_{j_1, \ldots, j_k, \infty}$, so that
	\begin{align*}
		&\mathfrak{h}_{j_1 \ldots, j_k} f = \Big\lVert\, f + \sum_{l=1}^k\, (-1)^l \sum_{1 \leq s_1 < \ldots < s_l \leq k}	T_{j_{s_1} \ldots j_{s_l}}f\, \Big\rVert \\
		&= \Big\lVert \sum_{i_1 \neq \ldots \neq i_d} \big( h(X_{i_1}, \ldots, X_{i_d}) + \sum_{l = 1}^k (-1)^l \sum_{s_1 < \ldots < s_l} T_{j_{s_1} \ldots j_{s_l}}h(X_{i_1}, \ldots ,X_{i_d}) \big) \Big\rVert
		\\ &\eqqcolon \Big\lVert \sum_{i_1 \neq \ldots \neq i_d} S_{i_1, \ldots, i_d}(h,X) \Big\rVert. 
	\end{align*}
Now it is easy to see that  $S_{i_1, \ldots, i_d}(h,X) = 0$ unless $\{j_1, \ldots, j_k \} \subset \{ i_1, \ldots, i_d \}$ (for example, this follows if one writes the sum inside the norm as $\prod_{i = 1}^k (\mathrm{Id} - T_{j_i}) f$), and in these cases one can upper bound the supremum by $2^k B$, from which we infer
	\begin{align*}
	\mathfrak{h}_{j_1 \ldots, j_k} f &\le \binom{d}{k} 2^k B (n-k) \cdots (n-d+1) \le \binom{d}{k} 2^k B n^{d-k}.
	\end{align*}
Consequently, this leads to
\[
	\norm{\mathfrak{h}^{(k)} f}_{\mathrm{HS},\infty} \le \binom{d}{k} 2^k B n^{d-k} n^{k/2} = \binom{d}{k} 2^k B n^{d-k/2}. 
\]

Thus, an application of Theorem \ref{theorem:Lpnormestimates} or \ref{theorem:partialLSITails} respectively yields for any $t \ge 0$ and for $C$ as given therein
\begin{align*}
	\IP \Big( \abs{f - \IE f} \ge t \Big) \le 2 \exp \Big( - \frac{1}{C} \min_{k=1,\ldots, d} \Big( \frac{t}{B \binom{d}{k}2^k n^{d-k/2}} \Big)^{2/k} \Big).
\end{align*}
For the second part, choose $t = B n^{d-1/2} \tilde{t}$ for $\tilde{t} > 0$ to obtain
\[
\IP \Big( n^{1/2 - d} B^{-1} \abs{f(X) - \IE f(X)} \ge t \Big) \le 2 \exp \Big( - \frac{1}{4C_d} \min_{k=1,\ldots, d} n^{\frac{k-1}{k}} t^{2/k} \Big).
\]
A short calculation shows that the minimum is attained for $k = 1$ in the range $t \le n^{1/2}$ and for $k = d$ otherwise, i.\,e.
\begin{align}\label{eqn:estimateForUStat}
\IP (n^{1/2 - d} B^{-1} \abs{f(X) - \IE f(X)} \ge t) \le 2 \exp \Big( - \frac{1}{4C_d} \min(t^2, n^{1-1/d} t^{2/d}) \Big).
\end{align}
\end{proof}

\begin{proof}[Proof of Theorem \ref{theorem:concentrationHLSI}]
We give a sketch of the proof only and refer to \cite[Proof of Theorem 2.2]{AKPS18} for details. Recall that by \eqref{eqn:fMinusEfLP} we have the inequality
\[
	\norm{f_d(X) - \IE f_d(X)}_p \le (4\sigma^2 p)^{1/2} \norm{\nabla f_d(X)}_p.
\]
Using the arguments and notations from \cite[Proof of Theorem 2.2]{AKPS18} leads to
\[
	\norm{f_d(X) - \IE f_d(X)}_p \le \sum_{k=1}^d (4\sigma^2 M^2)^{k/2} \norm{\skal{\IE_X \nabla^{(k)} f_d(X), G_1 \otimes \ldots \otimes G_k}}_p,
\]
where $M$ is an absolute constant and $G_i$ is a sequence of independent standard Gaussian random variables, independent of $X$. Furthermore, a result by Lata{\l}a \cite{La06} yields
\begin{align*}
	\norm{f_d(X) - \IE f_d(X)}_p &\le \sum_{k = 1}^d \sum_{\mathcal{I} \in P_k} (4\sigma^2 M^2 p)^{k/2} \norm{\IE \nabla^{(k)} f(X)}_{\mathcal{I}} \\
	&\le \sum_{k = 1}^d \sum_{\mathcal{I} \in P_k} (C \sigma^2 p)^{k/2} \norm{\IE \nabla^{(k)} f(X)}_{\mathcal{I}}.
\end{align*}
The rest now follows as in the previous proofs.
\end{proof}

\begin{proof}[Proof of Corollary \ref{corollary:ERGMTriangle}]
In \cite{SS18} the authors have proven that $\frac{1}{2} \Phi_{\abs{\bb}}'(1) < 1$ implies a $\dpartial\mathrm{-LSI}(\sigma^2)$ for $\mu_{\bb}$ with a constant depending on the parameter $\bb$ only. Thus, it remains to bound the norms in \eqref{eqn:partialLSIpolynomial}. Note that due to the structure of the exponential random graph model, the expectations of $\IE X_G$ and $\IE X_H$ are equal whenever $G$ and $H$ are isomorphic. Thus, we define $C_{S_2} \coloneqq \IE X_{S_2}$ (where $S_2$ is a $2$-star) and $C_{E} = \IE X_e$. \par
The Euclidean norms can be easily bounded:
\begin{align*}
\abs{\IE \nabla f(X)} &= \Big( \sum_{e \in \mathcal{I}_n} ((n-2) C_{S_2})^2 \Big)^{1/2} \le C_{S_2} n^2 \\
\abs{\IE \nabla^{(2)} f(X)} &= \Big( \sum_{e, f: e \cap f \neq \emptyset} \left( C_{E} \right)^2 \Big)^{1/2} \le C_E n^{3/2} \\
\abs{\IE \nabla^{(3)} f(X)} &= \Big( \sum_{\{e,f,g\} \in \mathcal{T}_3} 1 \Big)^{1/2} = n^{3/2},
\end{align*}
and it remains to estimate the three remaining norms. However, in \cite[Section 5.1]{AW15}, the authors given estimates for such norms in the Erd{\"o}s--R{\'e}nyi case, and it is easy to adapt these to any model with the property that $\IE X_G$ depends only on the isomorphism class of $G$ (in the complete graph). Especially, due to the structure of the exponential random graph models, this is true in this setting as well. This gives
\begin{align*}
\abs{\IE \nabla^{(3)} f(X)}_{\mathrm{op}} &\le 2^{3/2}, \quad 
\abs{\IE \nabla^{(3)} f(X)}_{\{1,2\}\{3\}} \le \sqrt{2n}, \quad
\abs{\IE \nabla^{(2)} f(X)}_{\mathrm{op}} \le 2C_E n.
\end{align*}
Inserting these estimates into \eqref{eqn:partialLSIpolynomial} finishes the proof.
\end{proof}

\section{Logarithmic Sobolev inequalities and difference operators}	\label{section:LSIs&DOs}

To conclude this paper, we discuss the LSI property \eqref{eqn:defiLSImitGamma} for different choices of difference operators $\Gamma$. Here, we always assume that the probability measure $\mu$ is defined on a product of Polish spaces $\mathcal{Y} = \otimes_{i=1}^n \mathcal{X}_i$ with product Borel $\sigma$-algebra $\mathcal{A} = \mathcal{B}(\otimes_{i = 1}^n \mathcal{X}_i)$.

In this situation, we can make use of the disintegration theorem on Polish spaces (see \cite[Chapter III]{DM78} and \cite[Theorem 5.3.1]{AGS08}): If $\mu$ is a measure on $\mathcal{Y}$, then for each $i \in \{1,\ldots,n\}$ we can decompose $\mu$ using the marginal measure $\mu_{i^c}$ (as a measure on $\otimes_{j \neq i} \mathcal{X}_i$) and a conditional measure on $\mathcal{X}_i$, which we denote by $\mu(\cdot \mid x_{i^c})$. More precisely, for any $A \in \mathcal{A}$ we have $\mu(A) = \int_{\otimes_{j \neq i} \mathcal{X}_i} \int_{\mathcal{X}_i} \eins_{A}(x_{i^c}, x_i) d\mu(x_i \mid x_{i^c}) d\mu_{i^c}(x_{i^c})$.

For finite spaces, $\mu(\cdot \mid x_{i^c})$ is just the ordinary conditional measure as used in the definition of the difference operator $\dpartial$. Note that the definition of $\dpartial$ can in principle be rewritten for products of arbitrary Polish spaces. However, our first result shows that the $\dpartial$-LSI property in fact requires the underlying space to be finite. More precisely, we say that $\mu$ has finite support if there is no sequence of sets $A_n \in \mathcal{A}$ with $\mu(A_n) > 0$ for any $n$ and $\mu(A_n) \to 0$.

\begin{proposition}\label{d-LSIs}
Let $\mathcal{Y} = \otimes_{i=1}^n \mathcal{X}_i$ be a product of Polish spaces, and let $\mu$ be a probability measure on $\mathcal{Y}$. If $\mu$ satisfies a $\dpartial$-LSI, then $\mu$ has finite support. Moreover, if $\mu$ is a product probability measure, then $\mu$ satisfies a $\dpartial$-LSI iff $\mu$ has finite support.
\end{proposition}

\begin{proof}
First assume $\mu$ does not have finite support, i.\,e. there is a sequence $A_n \in \mathcal{A}$ with $\mu(A_n) \to 0$. Choosing $f_n \coloneqq \eins_{A_n} \in L^\infty(\mu)$ and assuming a $\dpartial$-LSI$(\sigma^2)$ holds, we obtain
\begin{equation}\label{d-Widersp}
	\mu(A_n) \log(1/\mu(A_n)) = \Ent_{\mu}(f_n^2) \le 2 \sigma^2 \int (\dpartial f_n)^2 d\mu = 2 \sigma^2 \mu(A_n)(1-\mu(A_n)).
\end{equation}
This easily leads to a contradiction.

On the other hand, let $\mu$ be a product probability measure with finite support. By tensorization, it suffices to consider $n=1$, and we may moreover assume $\mathcal{Y}$ to have finitely many elements only. Then, by \cite[Remark 6.6]{BT06}, $\mu$ satisfies a $\dpartial$-LSI$(\sigma^2)$ with $\sigma^2 \le C \log(1/\min_{y : \mu(y) > 0} \mu(y))$, which finishes the proof.
\end{proof}

In fact, Proposition \ref{d-LSIs} can be adapted to the difference operator $\mathfrak{h}^+$ as well. To see this, note that that \eqref{d-Widersp} can easily be rewritten for the difference operator $\mathfrak{h}^+$ (with only minor changes) and $\int\abs{\dpartial f}^2d\mu \le \int\abs{\mathfrak{h}^+ f}^2d\mu$. In particular, the $\dpartial$- and $\mathfrak{h}^+$-LSI properties are not essentially different.

The situation drastically changes if we consider $\mathfrak{h}$-LSIs instead. Here, a sufficient condition for the $\mathfrak{h}\mathrm{-LSI}$ property to hold is that the measure $\mu$ satisfies an approximate tensorization (AT) property. As a consequence, for product probability measures, satisfying an $\mathfrak{h}$-LSI is in fact a universal property.

\begin{theorem}	\label{theorem:hLSIapproximatetensorization}
	Let $\mathcal{Y} = \otimes_{i = 1}^n \mathcal{X}_i$ be a product of Polish spaces, and let $\mu$ be a probability measure on $\mathcal{Y}$. If $\mu$ satisfies an approximate tensorization property
	\begin{align}
		\Ent_{\mu}(f^2) \le C \sum_{i = 1}^n \int \Ent_{\mu(\cdot \mid x_{i^c})}(f^2(x_{i^c}, \cdot)) d\mu_{i^c}(x_{i^c}),
	\end{align}
	then $\mu$ also satisfies an $\mathfrak{h}\mathrm{-LSI}(C)$. In particular, any product probability measure satisfies an $\mathfrak{h}\mathrm{-LSI}(1)$.
\end{theorem}

To the best of our knowledge, Theorem \ref{theorem:hLSIapproximatetensorization} is new. For product measures, it might be compared to the Efron--Stein inequality (see e.\,g. \cite{ES81,St86}) which establishes the tensorization property for the variance, and can be regarded as a universal Poincar{\'e} inequality with respect to $\dpartial$ (see e.\,g. \cite{BGS18} for such an interpretation). 
However, note that Theorem \ref{theorem:hLSIapproximatetensorization} (i.\,e. more precisely the $\mathfrak{h}\mathrm{-LSI}(1)$ for product measures) does not imply the Efron--Stein inequality, as the difference operator is $\mathfrak{h}$ instead of $\dpartial$. Unfortunately, as Proposition \ref{d-LSIs} demonstrates, there is no ``entropy version'' of the Efron--Stein inequality of the form $\Ent_{\mu}(f^2) \le C \IE_\mu \abs{\dpartial f}^2$ (for any product probability measure $\mu$ and some universal constant $C$).

As by Theorem \ref{theorem:hLSIapproximatetensorization}, any set of independent random variables $X_1, \ldots, X_n$ satisfies an $\mathfrak{h}$-LSI$(1)$, it might be tempting to regard Theorem \ref{theorem:Lpnormestimates} as an $\mathfrak{h}$-LSI analogue of Theorem \ref{theorem:partialLSITails}. However, it seems that it is not possible to use the entropy method based on $\mathfrak{h}$-LSIs, so that this interpretation is not fully accurate.
More precisely, Theorem \ref{theorem:hLSIapproximatetensorization} cannot be used to estimate the growth of $L^p$ norms as in the setting of a $\dpartial\mathrm{-LSI}(\sigma^2)$. Indeed, it is impossible to prove the required moment inequalities
\begin{align}		\label{eqn:LpInequalityUnderH}
	\norm{f -\IE f}_q \le (\sigma^2 q)^{1/2} \norm{\mathfrak{h}f}_q
\end{align}
under an $\mathfrak{h}\mathrm{-LSI}(\sigma^2)$. For example, the measure $\mu_p = p\delta_1 + (1-p)\delta_0$ satisfies $\mathfrak{h}\mathrm{-LSI}(\sigma^2_p)$ with $\sigma_p^2 \sim p(1-p) \log(1/p)$ (for $p \to 0$), so that \eqref{eqn:LpInequalityUnderH} would imply for $f(x) = x$ an upper bound on the Orlicz norm associated to $\Psi_2(x) = e^{x^2} - 1$
\begin{align*}	
	\norm{f - \IE f}_{\Psi_2} \le 2e \sup_{q \ge 1} \frac{\norm{f - \IE f}_q}{q^{1/2}} \le 4e\sigma_p.
\end{align*}
However, a simple calculation shows that $\IE \exp\big( \frac{(f-\IE f)^2}{16e^2\sigma_p^2} \big) \to \infty$ as $p \to 0$.

The approximate tensorization property in Theorem \ref{theorem:hLSIapproximatetensorization} is interesting in its own right, but it is not yet well-studied. For finite spaces \cite{Ma15} gives sufficient conditions for a measure $\mu$ to satisfy an approximate tensorization property. Similar results have been derived in \cite{CMT15}, which can be applied in discrete and continuous settings.
For example, if one considers a measure of the form $$\mu(x) = Z^{-1} \prod_{i =1}^n \mu_{0,i}(x_i) \exp\Big( \sum_{i,j} J_{ij} w_{ij}(x_i,x_j) \Big)$$ 
for some countable spaces $\Omega_i$, $x_i \in \Omega_i$, measures $\mu_{0,i}$ on $\Omega_i$ and bounded functions $w_{ij}$, under certain technical conditions $\mu$ satisfies an approximate tensorization property. This does not require any functional inequality for $\mu_{0,i}$. Very recently, in \cite[Proposition 5.4]{AKPS18} it has been shown that the $\mathrm{AT}(C)$ property implies dimension-free concentration inequalities for convex functions.

Note that the $\mathrm{AT}(C)$ property requires a certain weak dependence assumption in general. For example, the push-forward of a random permutation $\pi$ of $[n]$ to $\IN^{n}$ cannot satisfy an approximate tensorization property. It is an interesting question to find necessary and sufficient conditions for the approximate tensorization property to hold.

\begin{proof}[Proof of Theorem \ref{theorem:hLSIapproximatetensorization}]
	Let $X = (X_1, \ldots, X_n)$ be a $\mathcal{Y}$-valued random vector with law $\mu$. First we consider the case $n = 1$.
	By homogeneity of both sides, we may assume $\int f^2(X) d\IP = 1$. Since $f$ is bounded, we have $0 \le a \le \abs{f(X)} \le b < \infty$ $\IP$-a.s., where $b$ is the essential supremum of $\abs{f(X)}$ and $a$ the essential infimum. Due to the constraints on the integral this leads to $a^2 \le 1 \le b^2$. (Actually the cases $b = 1$ or $a = 1$ are trivial, since then $f^2(X) = 1$ $\IP$-a.s., but we will not make this distinction.) Let $F(u) \coloneqq \IP(f^2(X) \ge u)$. In particular
	\[
		F(u) = \begin{cases}
			1 & u \le a^2, \\
			0 & u > b^2.
		\end{cases}
	\]
	Using the partial integration formula (see e.\,g. \cite[Theorem 21.67 and Remark 21.68]{HS75}) in connection with \cite[Theorem 7.7.1]{Bu07} yields
	\begin{align*}
		\Ent(f^2(X)) &= \int_0^\infty u \log u d(-F(u)) = \int_0^{b^2} (\log u + 1) F(u) du \\ 
		&= \int_0^{a^2} (\log u + 1) F(u) du + \int_{a^2}^{b^2} (\log u +1) F(u) du \\
		&= \int_0^{a^2} (\log u + 1)F(u) du + \int_{a^2}^{b^2} \log u F(u) du + (1-a^2).
	\end{align*}
	The first integral can be calculated explicitly
	\begin{align*}
		\int_0^{a^2} (\log u + 1) F(u) du = u(\log u - 1)\mid_0^{a^2} = a^2 \log a^2,
	\end{align*}
	and moreover we have due to $\log(u) \le \log(b^2)$ on $[a^2, b^2]$
	\begin{align*}
	\int_{a^2}^{b^2} \log u F(u) du \le \log(b^2)(1-a^2).
	\end{align*}
	Plugging in these two estimates yields
	\begin{align*}
	\Ent(f^2(X)) \le a^2 \log a^2 + (1-a^2) + \log b^2 (1-a^2) \eqqcolon f(a,b).
	\end{align*}
	Next, if we show that 
	\begin{align}		\label{eqn:fleg}
	f(a,b) \le 2 (b-a)^2 \text{ on } G \coloneqq \{ (a,b) \in \IR^2 : 0 \le a \le 1, 1 \le b < \infty \}, 
	\end{align}
	we can further estimate (as $\abs{\mathfrak{h}f}^2$ is a deterministic quantity in the case $n = 1$)
	\begin{align*}
	\Ent(f^2(X)) \le 2(b-a)^2 = 2 \IE \abs{\mathfrak{h}f}^2.
	\end{align*}
	To prove \eqref{eqn:fleg}, define 
	\begin{align*}
	g(a,b) \coloneqq a^2 \log a^2 + (1-a^2) + \log b^2 (1-a^2) - 2(b-a)^2.
	\end{align*}
	Now it is easy to see that $g(a,1) = a^2 \log a^2 + (1-a^2) - 2(1-a)^2 \le 0,$ since $\partial_ag(a,1) \ge 0$ for $a \in [0,1]$ and $g(1,1) = 0$. Moreover
	\begin{align*}
	\partial_b g(a,b) = - \frac{2}{b}\left( b^2 - 1 + (a-b)^2 \right) \le 0,
	\end{align*}
	so that $g$ is decreasing on every strip $\{ a_0 \} \times [1,\infty)$, and thus $g(a,b) \le 0$ for all $a,b \in G$. This finishes the proof for $n=1$. \par
	For arbitrary $n$, the proof is now easily completed. Assume that $f \in L^\infty(\mu)$, i.\,e. $\mu_{i^c}(x_{i^c})$-a.s. we have $f(x_{i^c}, \cdot) \in L^\infty(\mu(\cdot \mid x_{i^c}))$. For these $x_{i^c}$, by the $n=1$ case we therefore obtain
	\begin{align*}
		\Ent_{\mu(\cdot \mid x_{i^c})}(f^2(x_{i^c}, \cdot)) \le 2 \sup_{y_i', y_i''} \abs{f(x_{i^c}, y_i') - f(x_{i^c}, y_i'')}^2.
	\end{align*}
	Plugging this into the assumption leads to
	\begin{align*}
		\Ent_{\mu}(f^2) \le 2C \int \sum_{i = 1}^n \sup_{y_i', y_i''} \abs{f(x_{i^c}, y_i') - f(x_{i^c}, y_i'')}^2 d\mu_{i^c}(x_{i^c}) = 2C \int \abs{\mathfrak{h}f}^2 d\mu.
	\end{align*}
  
  As for the second part, it is a classical fact that independent random variables satisfy the tensorization property (i.\,e. $\mathrm{AT}(1)$), see for example \cite[Proposition 5.6]{Led01}, \cite[Theorem 4.10]{BBLM05} or \cite[Theorem 3.14]{vH16}. In the case of independent random variables, the assumption that $\mathcal{Y}$ is a product of Polish spaces can be dropped by simply defining $\mu(\cdot \mid x_{i^c}) \coloneqq \mu_i = \IP \circ X_i$.
\end{proof}

\end{document}